\documentclass[11pt]{amsart}
\usepackage[english]{babel}
\usepackage[latin1]{inputenc}
\usepackage{amsmath,amsfonts,amssymb,amsthm,amscd,array,stmaryrd,mathrsfs}
\usepackage[table]{xcolor}
\usepackage{collectbox}
\usepackage{MnSymbol}
\usepackage{float}
\makeatletter
\newcommand{\mybox}{%
    \collectbox{%
        \setlength{\fboxsep}{1pt}%
        \fbox{\BOXCONTENT}%
    }%
}
\makeatother

\usepackage{url}
 \usepackage[all]{xy}
\usepackage{graphicx}
\usepackage{color}              
\usepackage{epsfig}
\usepackage{psfrag}
\usepackage{epstopdf}

\usepackage{textcomp}

\theoremstyle{plain}
\newtheorem{lem}{Lemma}[section]
\newtheorem{thm}[lem]{Theorem}
\newtheorem{cor}[lem]{Corollary}
\newtheorem{prop}[lem]{Proposition}
\newtheorem{conj}[lem]{Conjecture}

\theoremstyle{definition}
\newtheorem{defi}[lem]{Definition}
\newtheorem{rem}[lem]{Remark}
\newtheorem{ex}[lem]{Example}

\DeclareMathOperator*{\rank}{rank}

\newcommand{\C}{\mathbb{C}}

\newcommand{\T}{\mathcal{T}}

\newcommand{\F}{\mathcal{F}}

\newcommand{\GL}{\mathrm{GL}}

\newcommand{\Lc}{\mathcal{L}} 
 
\newcommand{\X}{\mathcal{X}}

\newcommand{\gn}{\mathfrak{n}}

\newcommand{\tr}{\textup{trc}}
\newcommand{\trp}{\textup{trp}}

\newcommand{\Id}{\textup{Id}}
\newcommand{\e}{\varepsilon}

\newcommand{\codim}{\mathrm{codim }}

\def\e{\varepsilon}

\begin{document}

\title{On tangent cones of Schubert varieties}

\author{D. Fuchs,
A. Kirillov,
S. Morier-Genoud,
V. Ovsienko }

\address{Dmitry Fuchs,
Mathematical Sciences Building 
One Shields Ave.
University of California
Davis, CA 95616
fuchs@math.ucdavis.edu
}

\address{Alexandre Kirillov,
Department of Mathematics,
209 South 33rd Street,
University of Pennsylvania
Philadelphia, PA 19104-6395,
kirillov@math.upenn.edu}

\address{Sophie Morier-Genoud,
Sorbonne Universit\'es, UPMC Univ Paris 06, 
Institut de Math\'ematiques de Jussieu-Paris Rive Gauche, UMR 7586, CNRS, 
Univ Paris Diderot, Sorbonne Paris Cit\'e, F-75005, Paris, France,
sophie.morier-genoud@imj-prg.fr
}

\address{
Valentin Ovsienko,
CNRS,
Laboratoire de Math\'ematiques 
U.F.R. Sciences Exactes et Naturelles 
Moulin de la Housse - BP 1039 
51687 REIMS cedex 2,
France,
valentin.ovsienko@univ-reims.fr}

\subjclass[2010]{}
\date{}

\begin{abstract}
We consider tangent cones of Schubert varieties in the complete flag variety,
and investigate the problem when the tangent cones of two different Schubert varieties
coincide.
We give a sufficient condition for such coincidence, and formulate a conjecture
that provides a necessary condition.
In particular, we show that all Schubert varieties corresponding to the Coxeter elements
of the Weyl group have the same tangent cone.
Our main tool is the notion of pillar entries in the rank matrix counting the
dimensions of the intersections of a given flag with the standard one.
This notion is a version of Fulton's essential set.
We calculate the dimension of a Schubert variety in terms of the pillar entries
of the rank matrix.
\end{abstract}
\maketitle

\thispagestyle{empty}

\tableofcontents

\section{Introduction}

Let $\F$ be the algebraic variety of all complete flags in $\C^n$. Recall that a complete flag $F\in\F$ is an increasing sequence of subspaces
\begin{equation*}
\label{FlagEq}
\{0\}=V_0\subset{}V_1\subset{}V_2\subset\cdots\subset{}V_n=\C^n,
\qquad\dim{}V_k=k.
\end{equation*}
Choosing the standard basis $\{\e_1,\ldots,\e_n\}$ of $\C^n$, one defines the standard flag, $F_0\in\F$, for which $V_k=\C^k:=\langle\e_1,\ldots\e_k\rangle$, for all $1\leq{}k\leq{}n$. The group $\GL(n,\C)$ of linear transformations of $\C^n$ 
transitively acts on $\F$. The Borel subgroup $B\subset\GL(n,\C)$ of upper-triangular matrices is the stabilizer of the standard flag $F_0$, so $\F=\GL(n,\C)/B$.

Let us recall some well-known facts. The group $B$ acts naturally on $\F$ (by left multiplication). The variety $\F$ is a disjoint union of $B$-orbits called {\it Schubert cells}. Schubert cells are indeed cells of the most classical CW  decomposition of $\F$. Schubert cells are parametrized by elements of the symmetric group~$S_n$. Namely, the group~$S_n$ acts naturally in $\C^n$, and hence in $\F$, and for every $w\in S_n$, there exists a unique Schubert cell, which contains the $w$-image of the standard flag $F_0$. We denote this cell by ${\mathcal C}_w$. Its complex dimension is equal to the {\it length} of $w$,  i.e., the minimal $\ell$ in a decomposition$$w=s_{i_1}s_{i_2}\cdots{}s_{i_\ell},$$where $s_i\in{}S_n$ are the elementary transpositions. The number of Schubert cells of complex dimension $m$ is the coefficient at $t^m$ in the polynomial 
$$
\prod_{k=1}^n(1+t+\ldots+t^k).
$$ 
In particular, there is a unique 0-dimensional cell, which is $F_0$, and a unique 
$\frac{n(n-1)}2$-dimensional cell, which is dense in $\F$.

The closure $\X_w$ of a Schubert cell ${\mathcal C}_w$ is called a {\it Schubert variety}. The Schubert variety $\X_w$ is the union of the Schubert cell ${\mathcal C}_w$ and all Schubert cells 
${\mathcal C}_{w'}$ corresponding to permutations~$w'$ which precede $w$ with respect to the natural partial ordering of $S_n$ (the Bruhat order). In particular, every Schubert variety contains the point $F_0$.

With a Schubert variety $\X_w$, we associate two subsets of the tangent space $T_{F_0}\F$: 
\begin{itemize}
\item
the {\it tangent cone} $\T_w$, which is the set of vectors tangent to $\X_w$ at $F_0$;
\item
 the Zariski {\it tangent space} $\mathcal{Z}_w$ which is {\it spanned} by $\T_w$. 
 \end{itemize}

 The tangent cones $\T_w$ are algebraic subvarieties of $T_{F_0}\F$; they have the same dimensions as $\X_w$ (and ${\mathcal C}_w$). The tangent cone $\T_w$ and tangent space ${\mathcal Z}_w$ (as well as their dimensions) coincide if and only if $F_0$ is not a singular point of $\X_w$. 

Certainly, the Schubert varieties $\X_w$ and $\X_{w'}$ coincide only when $w=w'$; however, the equalities ${\mathcal Z}_w={\mathcal Z}_{w'}$ or $\T_w=\T_{w'}$ may occur for $w\neq w'$ (since the second implies the first, the first occurs ``more often" than the second). 

For the further discussion, let us introduce the most natural local coordinate system in a (Zariski) neighborhood of $F_0$ in $\F$. For a flag $\{V_k\}$ sufficiently ``close" to $F_0$, there exists a unique ``triangular" basis in $\C^n$,
$$
v_1=
\left(
\begin{array}{l}
1\\
x_{21}\\
x_{31}\\
\vdots\\
x_{n1}
\end{array}
\right),
\qquad
v_2=
\left(
\begin{array}{l}
0\\
1\\
x_{32}\\
\vdots\\
x_{n2}
\end{array}
\right),
\quad\dots,\qquad
v_n=
\left(
\begin{array}{l}
0\\
0\\
\vdots\\
0\\
1
\end{array}
\right)
$$
such that $V_k$ is spanned by $v_1,\dots,v_k$. The numbers $x_{ij},i>j$, are coordinates of the flag $\{V_k\}$ (with $F_0=(0,\dots,0)$); the same numbers may be regarded as coordinates in $T_{F_0}\F$. (This coordinate system provides a natural identification of $T_{F_0}\mathcal F$ with the space ${\mathfrak n}_-$ of strictly lower triangular matrices.) When $n$ is not too large, we will use the more convenient notations $x_i=x_{i,i+1},y_i=x_{i,i+2}$, etc. 

Zariski tangent spaces ${\mathcal Z}_w$ were thoroughly studied, see \cite{Pol, Lak1,Lak} and references therein. The following result of Lakshmibai~\cite{Lak1} provides an explicit description of ${\mathcal Z}_w$. The space~${\mathcal Z}_w$, viewed as a subspace of ${\mathfrak n}_-$, is the linear span of the elements $e_{-\alpha}$ of the Chevalley basis, such that 
$$
\alpha\in R^+,\; s_\alpha\leq w,
$$ 
where $R^+$ is the set of positive roots,
and $s_\alpha\in S_n$ is the reflection associated with $\alpha$, and $\leq$ is the Bruhat order. 
The above result, of course, answers the question, under which condition two different Schubert varieties $\X_w$ and $\X_{w'}$ have the same Zariski tangent space. On the contrary, the structure of tangent cones~$\T_w$, although it has been an active area of research (see~\cite{Lak,Bri,CK,EP,BIS} and references therein), is not well understood, in particular, the problem of their coincidence is mostly open.

Let us consider some examples. If $n=3$, then $\dim\F=3$ and the local coordinates are $x_1,x_2,y$. There are 6 Schubert varieties of dimensions $0,1,1,2,2,3$, and the middle four are: 
$$
\X_{213}=\{V_1=\C^1\},\quad
 \X_{213}=\{V_2=\C^2\}, \quad
 \X_{231}=\{V_1\subset\C^2\},\quad
 \X_{312}=\{V_2\supset\C^1\}.
 $$
 In our local coordinates these are 
 $x_1=y=0, \,
 x_2=y=0, \,
 y=0,\,
 y=x_1x_2$, respectively. 
 We see that, within the domain of our coordinate system, $\X_{231}$ is the tangent plane (at the origin) to $\X_{312}$; thus $\T_{231}=\T_{312}={\mathcal Z}_{231}={\mathcal Z}_{312}$.

The first examples of singular Schubert varieties appear when $n=4$. There are two of them, 
cf.~\cite{LS}: 
$$
\X_{3412}=\{V_1\subset\C^3,\, \C^1\subset V_3\}
\quad\hbox{and}\quad
\X_{4231}=\{V_2\cap\C^2\neq0\}.
$$
Our local coordinates in the 6-dimensional manifold $\F$ are $x_1,x_2,x_3,y_1,y_2,z$, the equations of the two Schubert varieties are 
$$
z=0,\quad y_1x_3+x_1y_2-x_1x_2x_3=0
\qquad\hbox{and}
\qquad
y_1y_2-zx_2=0,
$$ 
 respectively, and the tangent cones are the cone $y_1x_3+x_1y_2=0$ in the hyperplane $z=0$ and the cone $y_1y_2-zx_2=0$ in the whole space $T_{F_0}\F$. It is not difficult to observe that the~$24$ Schubert varieties have $16$ different tangent cones and $14$ different tangent spaces.

For $n=5$, we observe not only singular, but also reducible tangent cones (the Schubert varieties themselves are always irreducible). Moreover, different tangent cones can share components and even contain each other. The simplest example
is provided by the 8-dimensional Schubert varieties 
$$
\X_{35421}=\{V_1\subset\C^3\},
\quad 
\X_{43521}=\{V_2\subset\C^4\}
\quad\hbox{and}\quad
\X_{45231}=\{V_1\subset\C^4,\C^2\cap V_3\neq0\}.
$$
With respect to the local coordinates $x_1,x_2,x_3,x_4,y_1,y_2,y_3,z_1,z_2,t$, the first two varieties (and hence their tangent cones) are linear subspaces $z_1=t=0$ and $z_2=t=0$, while the third one is described by the equations $t=0, \det\left[\begin{array} {ccc} y_1&x_2&1\\ z_1&y_2&x_3\\ 0&z_2&y_3\end{array}\right]=0$. This shows that the tangent cone $\T_{45231}$ is $\{t=z_1z_2=0\}$, and this is the union $\T_{35421}\cup\T_{43521}$.

In this paper, we study the structure of the tangent cones $\T_w$ with the emphasis on the problem of their coincidence. Let us mention two cases when the coincidence of these tangent cones is known, or can be easily proved. The first one is the equality $\T_w=\T_{w^{-1}}$ which holds for every permutation $w$. This fact was conjectured (and checked for $n\le5$) in~\cite{EP}; however, a short direct proof can be easily given, see Section~\ref{proofdirect}. The second case is that of {\it Coxeter elements} of the permutation group. Recall that an element $w\in S_n$ is called a Coxeter element, if it is of length $n-1$ and can be written in the form$$w=s_{i_1}s_{i_2}\cdots{}s_{i_{n-1}}$$in such a way that every transposition $s_i$, for $i=1,2,\ldots,n-1$ enters the above product exactly once. The group $S_n$ has $2^{n-2}$ different Coxeter elements. 
The Schubert varieties which correspond to the Coxeter elements of $S_n$ have the same tangent cone, namely the one given by the equations
$$
x_{ij}=0,
\qquad\hbox{for}\quad
i-j>1.
$$
By the way, our example of coinciding tangent cones for $n=3$ represents both cases: the permutations 
$132$ and $321$ are Coxeter elements inverse to each other. For $n=4$, all pairs of permutations with equal tangent cones are either Coxeter, or inverse to each other. However, for $n=5$, there appear pairs of non-inverse and non-Coxeter permutations with equal tangent cones; the first example of such a pair is $(13452,13524)$.

We develope an efficient method to recognize when the tangent cones of two Schubert varieties coincide. The main ingredient of this method is the notion of a {\it pillar entry}. Every Schubert cell of the flag variety is determined by the $(n+1)\times(n+1)$ matrix of dimensions~$r_{ij}$ of the intersections  $V_i\cap\C^j$ called the {\it rank matrix\/}; 
the corresponding Schubert variety is determined by inequalities $\dim(V_i\cap\C^j)\ge r_{ij}$. 
For example, if $\left[r_{ij}\right]$ is the rank matrix corresponding to a permutation $w$, then the rank matrix corresponding to $w^{-1}$ is obtained from~$\left[r_{ij}\right]$ by a transposition. 
In Section~5.6, we prove that the whole matrix $\left[r_{ij}\right]$ is determined by a relatively small set of entries, which we call {\it pillar entries} (see Section~\ref{PilDeSect} for a precise definition). 
Note that the notion of pillar entry is very close (yet different from) Fulton's notion of essential set~\cite{Ful1}, see also~\cite{EL,Woo,RWY} and the Appendix for a comparison.

We conjecture that if $\T_w=\T_{w'}$, then the pillar entries for $w'$ are obtained from pillar entries for $w$ by a {\it partial transposition}.
This means tat there is a one-to-one correspondence between pillar entries
 $r_{ij}$ and $r'_{ij}$ for $w$ and $w'$ such that the pillar entry corresponding to $r_{ij}$ is
either $r'_{ij}=r_{ij}$ or $r'_{ji}=r_{ij}$; see Section~\ref{CoSec}, Conjecture~\ref{TheConjThm} for a precise statement. However, the converse of this conjecture is false: examples show that a partial transposition of the set of pillar entries may lead to a set of entries which is not the set of pillar entries for any transposition, or is a set of pillar entries of a transposition of a different length. 
Some pillar entries are ``linked," that is, they can be transposed or not transposed only simultaneously. 

In Section~\ref{FMTSec}, we give some definition of a linkage, and hence of ``admissible partial transposition"; our main result is Theorem~\ref{TheThm}, which states that an admissible partial transposition of pillars entries of $w$ provides a set of pillar entries of some $w'$, 
and that in this case~$\T_w=\T_{w'}$. 
However, examples show that our definition of linkage is not sufficient: there are partial transpositions of pillar entries, which are not admissible in our sense, but which still preserve the tangent cone.

In Section~\ref{CombSect}, we study combinatorics of rank matrices and pillar entries.
In particular, we present a formula (see Theorem~\ref{codimDBF}) of (co)dimension of a Schubert variety
in terms of the pillar entries of the corresponding rank matrix.
We also present an algorithm that reconstructs a given permutation from
the corresponding pillar entries.

We also provide a number of examples and several enumerative results in small dimension and codimension.
We were led by the numeric examples to the following ``$2^m$-conjecture''
which is also closely related with the earlier mentioned conjecture:
{\it the number of Schubert varieties with an identical tangent cone is always a power of $2$}.

Let us  mention that the problem of classification of tangent cones of Schubert varieties is closely related to the problem of classification of coadjoint orbits of the unitriangular group, see~\cite{Kir,And} and the recent work~\cite{Pan}. As we already said, the tangent space to the flag variety is naturally identified with the nilpotent Lie algebra of lower-triangular matrices, and with the dual space of the Lie algebra of upper-triangular matrices: $$T_{F_0}\F\simeq\gn_{-}\simeq\gn_{+}^*.$$ The $B$-action on $T_{F_0}\F$ then coincides with the coadjoint action. Every tangent cone $\T_w$ is $B$-invariant, as well as any irreducible component of $\T_w$; thus, it is a set of $B$-orbits. However, it is not true that $B$-orbits and irreducible components of tangent cones are the same thing. The first example which demonstrates this appears in $S_6$: the $10$-dimensional tangent cone $\T_{354621}$ is a union of $9$-dimensional $B$-orbits. We will not discuss this phenomenon in this paper.

\section{Basic notions and main conjecture}

We recall the classical notion (and some properties) of rank matrix 
associated with two flags.
Rank matrices
provide a combinatorial way to characterize
Schubert varieties and Schubert cells.
Indeed, one of these flags will be chosen as the standard flag,
so that the rank matrix coincides with the rank function of
the corresponding permutation; see~\cite{Ful,Ful1}.
We then define the notion of pillar entry of a rank matrix
which is crucial for us.

We formulate our first conjecture that if two permutations, $w$ and $w'$,
have identical tangent cones: $\T_w=\T_{w'}$,
then the pillar entries of the corresponding rank matrices
either coincide or transposed to each other.

\subsection{Rank matrix}\label{rkmat}
For any flag, the {\it rank matrix} is the $(n+1)\times(n+1)$ matrix $r=(r_{ij})$
with the integer entries
$$
r_{ij}=\dim{}V_i\cap\C^j,
\qquad
0\leq{}i,j\leq{}n.
$$
The rank matrix is independent of the choice of
a flag in a $B$-orbit.
Moreover, it
completely characterizes the corresponding $B$-orbit.
More precisely, two different flags, $F\in{}{\mathcal C}_w$ and $F'\in{}{\mathcal C}_{w'}$,
have the same rank matrix if and only if $w=w'$; see, e.g.,~\cite{Ful}.
We will denote by~$r(w)$ the rank matrix corresponding
to the Schubert cell ${\mathcal C}_w$.

Obviously, one has:
$$
\begin{array}{l}
r_{0k}=r_{k0}=0;
\qquad
 r_{kn}=r_{nk}=k;\\[4pt] 
 r_{ij}+r_{i+1,j+1}\geq r_{i+1,j}+r_{i,j+1};\\[4pt] 
  r_{i,j+1}- r_{ij}=0\quad\hbox{or}\quad1;\\[4pt] 
  r_{i+1,j}- r_{ij}=0\quad\hbox{or}\quad1. 
  \end{array}
 $$
 Every integer matrix $(r_{ij})$ with the above properties is the rank matrix of some flag.

The following statement is due to~\cite{Ful1}, see also~\cite{Ful} p.157.
The Schubert cell ${\mathcal C}_{w}$ consists in flags such that
the corresponding rank matrix is:
\begin{equation}
\label{Ful}
r_{ij}=\#\{k\leq{}i\;\vert\; w(k)\leq{}j\}.
\end{equation}

\begin{ex}
\label{InvCor}
The rank matrices $r(w)$ and~$r(w^{-1})$ are
transposed to each other.
In this case, one has: 
$$
\T_w=\T_{w^{-1}}.
$$
This statement
 was conjectured
(and checked for $n\leq5$) in~\cite{EP}.
However a short direct proof can be easily given, see Section \ref{proofdirect}. \end{ex}

\subsection{Permutation diagram}

The permutation $w\in{}S_n$ can be easily recovered 
from the rank matrix.

\begin{defi}
Given a permutation $w \in S_{n}$, the \textit{diagram of } $w$
is defined with the following convention.
In an $(n+1)\times(n+1)$ grid, with row and columns numbered form $0$ to $n$, we place a dot in the upper left corner of the cell with coordinates $(i,j)$ whenever $j=w(i)$. 
\end{defi}

To make this visible, we usually put  a $\bullet$ into the matrix,
so that the permutation is encoded by the dots.

\begin{prop}
If the rank matrix $r(w)$ is locally as follows:
$$
\begin{array}{c|c}
a&a\\
\hline
a&\!\!\!\!{}^{{}^\bullet}a+1
\end{array}
$$
where $a+1$ is the value in position $(i,j)$,
then the permutation $w$ sends $i$ to $j$.
\end{prop}

\begin{proof}
This readily follows from~(\ref{Ful}).
\end{proof}

\begin{ex}
\label{Exn4}
Consider the case of dimension $4$.

a) The matrices
$$
\begin{array}{|c|c|c|c|c|}
\hline
0&0&0&0&0\\
\hline
0&0&0&0&\!\!\!\!{}^{{}^\bullet}\,1\\
\hline
0&0&0&\!\!\!\!{}^{{}^\bullet}\,1&2\\
\hline
0&0&\!\!\!\!{}^{{}^\bullet}\,1&2&3\\
\hline
0&\!\!\!\!{}^{{}^\bullet}\,1&2&3&4\\
\hline
\end{array}
\qquad
\begin{array}{|c|c|c|c|c|}
\hline
0&0&0&0&0\\
\hline
0&\!\!\!\!{}^{{}^\bullet}\raisebox{.3pt}{\textcircled{\raisebox{-.9pt} {1}}}\!\!\!&1&1&1\\
\hline
0&1&\!\!\!\!{}^{{}^\bullet}\raisebox{.3pt}{\textcircled{\raisebox{-.9pt} {2}}}\!\!\!&2&2\\
\hline
0&1&2&\!\!\!\!{}^{{}^\bullet}\raisebox{.3pt}{\textcircled{\raisebox{-.9pt} {3}}}\!\!\!&3\\
\hline
0&1&2&3&\!\!\!\!{}^{{}^\bullet}\,4\\
\hline
\end{array}
$$
are the rank matrices corresponding to the longest element $w_0=4321$ and the identity element $w=1234$, respectively.

The encircled entries will be later called ``pillar'', these entries determine
the whole matrix, as explained in the next subsection.

b) The following matrices:
$$
\begin{array}{|c|c|c|c|c|}
\hline
0&0&0&0&0\\
\hline
0&0&\!\!\!\!{}^{{}^\bullet}\raisebox{.3pt}{\textcircled{\raisebox{-.9pt} {1}}}\!\!\!&1&1\\
\hline
0&0&1&\!\!\!\!{}^{{}^\bullet}\raisebox{.3pt}{\textcircled{\raisebox{-.9pt} {2}}}\!\!\!&2\\
\hline
0&0&1&2&\!\!\!\!{}^{{}^\bullet}\,3\\
\hline
0&\!\!\!\!{}^{{}^\bullet}\,1&2&3&4\\
\hline
\end{array}
\qquad
\begin{array}{|c|c|c|c|c|}
\hline
0&0&0&0&0\\
\hline
0&0&\!\!\!\!{}^{{}^\bullet}\raisebox{.3pt}{\textcircled{\raisebox{-.9pt} {1}}}\!\!\!&1&1\\
\hline
0&0&1&1&\!\!\!\!{}^{{}^\bullet}\,2\\
\hline
0&\!\!\!\!{}^{{}^\bullet}\,1&\!\!\raisebox{.3pt}{\textcircled{\raisebox{-.9pt} {2}}}\!\!\!&2&3\\
\hline
0&1&2&\!\!\!\!{}^{{}^\bullet}\,3&4\\
\hline
\end{array}
\qquad
\begin{array}{|c|c|c|c|c|}
\hline
0&0&0&0&0\\
\hline
0&0&0&\!\!\!\!{}^{{}^\bullet}\,1&1\\
\hline
0&\!\!\!\!{}^{{}^\bullet}\raisebox{.3pt}{\textcircled{\raisebox{-.9pt} {1}}}\!\!\!
&1&\!\!\raisebox{.3pt}{\textcircled{\raisebox{-.9pt} {2}}}\!\!\!&2\\
\hline
0&1&1&2&\!\!\!\!{}^{{}^\bullet}\,3\\
\hline
0&1&\!\!\!\!{}^{{}^\bullet}\,2&3&4\\
\hline
\end{array}
\qquad
\begin{array}{|c|c|c|c|c|}
\hline
0&0&0&0&0\\
\hline
0&0&0&0&\!\!\!\!{}^{{}^\bullet}\,1\\
\hline
0&\!\!\!\!{}^{{}^\bullet}\raisebox{.3pt}{\textcircled{\raisebox{-.9pt} {1}}}\!\!\!
&1&1&2\\
\hline
0&1&\!\!\!\!{}^{{}^\bullet}\raisebox{.3pt}{\textcircled{\raisebox{-.9pt} {2}}}\!\!\!&2&3\\
\hline
0&1&2&\!\!\!\!{}^{{}^\bullet}\,3&4\\
\hline
\end{array}
$$
are the rank matrices corresponding to the four Coxeter elements in $S_4$: 
$$
s_1s_2s_3=2341,
\qquad
s_1s_3s_2=2413,
\qquad
s_2s_1s_3=3142,
\qquad
s_3s_2s_1=4123,
$$
respectively.

c)
Consider the elements $w_1=3412$ and $w_2=4231$ of $S_4$. 
The corresponding rank matrices are
$$
\begin{array}{|c|c|c|c|c|}
\hline
0&0&0&0&0\\
\hline
0&0&0&\!\!\!\!{}^{{}^\bullet}\raisebox{.3pt}{\textcircled{\raisebox{-.9pt} {1}}}\!\!\!&1\\
\hline
0&0&0&1&\!\!\!\!{}^{{}^\bullet}\,2\\
\hline
0&\!\!\!\!{}^{{}^\bullet}\raisebox{.3pt}{\textcircled{\raisebox{-.9pt} {1}}}\!\!\!&1&2&3\\
\hline
0&1&\!\!\!\!{}^{{}^\bullet}\,2&3&4\\
\hline
\end{array}
\qquad\hbox{and}\qquad
\begin{array}{|c|c|c|c|c|}
\hline
0&0&0&0&0\\
\hline
0&0&0&0&\!\!\!\!{}^{{}^\bullet}\,1\\
\hline
0&0&\!\!\!\!{}^{{}^\bullet}\raisebox{.3pt}{\textcircled{\raisebox{-.9pt} {1}}}\!\!\!&1&2\\
\hline
0&0&1&\!\!\!\!{}^{{}^\bullet}\,2&3\\
\hline
0&\!\!\!\!{}^{{}^\bullet}\,1&2&3&4\\
\hline
\end{array}
$$
The Schubert varieties $\X_{w_1}$ and $\X_{w_2}$ 
are the only singular Schubert varieties for $n=4$.
\end{ex}

\begin{ex}
For the maximal cell ${\mathcal C}_{w_0}$, the rank matrix is given by:
$$
r_{ij}(w_0)=
\max\{0,\,i+j-n\}.
$$
\end{ex}

The smaller is the Schubert cell ${\mathcal C}_w$, the bigger are the numbers
 $r_{ij}(w)$.

\subsection{The pillar entries}\label{PilDeSect}

The rank matrix is completely determined
by a few particular entries.
This idea is due to Fulton~\cite{Ful} (see also~\cite{Woo,RWY} and references therein).
The following notion is crucial for us.

\begin{defi}
\label{PilDef}
An entry $r_{ij}$ of a rank matrix $r(w)$ is called
{\it pillar} if it satisfies the conditions
\begin{equation}
\label{LocPil}
\left\{
\begin{array}{l}
r_{ij}=r_{i-1,j}+1=r_{i,j-1}+1,\\[6pt]
r_{ij}=r_{i+1,j}=r_{i,j+1}.
\end{array}
\right.
\end{equation}
\end{defi}

\noindent
In other words,
the fragment of the rank matrix around a pillar
entry is as follows:
$$
\begin{array}{ccc}
&a-1&\\[2pt]
a-1&\raisebox{.3pt}{\textcircled{\raisebox{-.1pt} {$a$}}}&a\\[2pt]
&a&
\end{array}
$$
We always encircle the pillar entries, in order to distinguish them.

In combinatorial terms, pillar entries can be characterized as follows.
An entry $r_{ij}$ of a rank matrix $r(w)$ is pillar if and only if
\begin{equation}
\label{ComPil}
\left\{
\begin{array}{ll}
w(i)\leq{}j,&w(i+1)>j,\\[6pt]
w^{-1}(j)\leq{}i,&w^{-1}(j+1)>i.
\end{array}
\right.
\end{equation}
It is easy to see that these conditions are equivalent to~(\ref{LocPil}).

It worth noticing that, the more a given permutation $w$ is ``close'' to the identity,
the more pillar entries the matrix $r(w)$ has.
The matrix $r(\Id)$ has $n-1$ pillar entries $r_{ii}=i$, for $1\leq{}i\leq{}n-1$. 
The more $w$ is ``close'' to the longest element $w_0$,
the less pillar entries the matrix $r(w)$ has.
In particular, $r(w_0)$ is the only rank matrix with no pillar entries.

\begin{prop}
\label{PiP}
Every Schubert cell is completely determined by the pillar entries
of the rank matrix.
\end{prop}

This statement is classical.
For the sake of completeness,
a proof will be presented in Section~\ref{ProProS}.
An explicit algorithm that reconstructs the permutation $w$ from the
pillar entries of the rank matrix~$r(w)$ will be presented
in Section~\ref{wfrompil}.

Let us describe the pillar entries of the rank matrices
corresponding to the Coxeter elements.

\begin{prop}
\label{ExCoxSec}
The rank matrix of any Coxeter element of $S_n$
has $n-2$ pillar entries 
$$r_{i,i+1}=i,
\qquad\hbox{or}\qquad
r_{i+1,i}=i,
$$
for each $i\in\{1,2,\ldots,n-2\}$.
\end{prop}

\begin{proof}
Consider a Coxeter element
$w=\cdots\,s_i\,\cdots\,s_{i+1}\,\cdots{}$.
It can be deduced directly from~(\ref{ComPil}), that the entry
$r_{i,i+1}$ of $r(w)$ is, indeed, a pillar entry.
Similarly, for a Coxeter element of the form
$w=\cdots\,s_{i+1}\,\cdots\,s_{i}\,\cdots{}$,
one has that the entry $r_{i+1,i}$ is pillar.
Similar arguments show that the rank matrix of a Coxeter element
cannot have other pillar entries than the above $n-2$ ones.

Finally, the fact that the value of the pillar entry $r_{i,i+1}$ (or $r_{i+1,i}$) is equal to $i$
follows from~(\ref{Ful}).
\end{proof}

\begin{rem}
In other words, the rank matrix of every Coxeter element of $S_n$
is determined by a sequence of $n-2$ inclusions:
$$
V_i\subset{}\C^{i+1},
\qquad\hbox{or}\qquad
\C^i\subset{}V_{i+1},
$$
for $i\in\{1,\ldots,n-2\}$.
The $2^{n-2}$ Coxeter elements correspond to an arbitrary
choice of one of the above inclusions for every $i$.
\end{rem}

We believe that the notion of pillar entry deserve a further study.
In particular, the number of pillar entries for a given permutation
is an interesting characteristic.
Some of the basic properties of pillar entries 
will be presented in Section~\ref{CombSect}.

\subsection{Fulton's essential entries}\label{CompSec}

Let us recall here Fulton's notion of essential entry.
An entry $r_{ij}$ of a rank matrix $r(w)$ is called {\it essential},
see~\cite{Ful1} and also~\cite{EL}, if
\begin{equation}
\label{ComEss}
\left\{
\begin{array}{ll}
w(i)>j,&w(i+1)\leq{}j,\\[6pt]
w^{-1}(j)>i,&w^{-1}(j+1)\leq{}i.
\end{array}
\right.
\end{equation}
Equivalently, the rank matrix around an essential
entry is as follows:
$$
\begin{array}{ccc}
&a&\\[2pt]
a&\raisebox{.3pt}{\boxed{\raisebox{-.1pt} {$a$}}}&a+1\\[2pt]
&a+1&
\end{array}
$$
It is proved in~\cite{Ful1} that
every rank matrix (and therefore the corresponding Schubert variety)
is completely characterized by its essential set.

The notions of essential and pillar entries are somewhat
``complementary'', as the inequality signs in formulas~(\ref{ComPil}) and (\ref{ComEss}) 
are reversed, cf. Appendix for a comparison.

\subsection{Transposed pillars: the main conjecture}\label{CoSec}

The following conjecture asserts that if two Schubert varieties
have the same tangent cones, then they have the same number of pillars,
whose values are also the same, and whose position in the respective rank matrices
can only differ by transposition.

\begin{conj}
\label{TheConjThm}
Given two permutations, $w$ and $w'\in{}S_n$, if
$\T_w=\T_{w'}$ then the rank matrices~$r(w)$ and~$r(w')$
have the same number of pillar entries, and
for every pillar entry~$r_{ij}$ of~$r(w)$,
one has the following alternative:

a) the entry $r'_{ij}$ of~$r(w')$ is pillar and $r'_{ij}=r_{ij}$, or

b) the entry $r'_{ji}$ of~$r(w')$ is pillar and $r'_{ji}=r_{ij}$.
\end{conj}

Example~\ref{InvCor} and Proposition~\ref{ExCoxSec} are the first
examples that confirm our conjecture.
We will give many other examples in the sequel.

\subsection{Restrictions: forbidden transpositions}\label{InvCoSec}

Note that the inverse of Conjecture~\ref{TheConjThm} is false:
two permutations with partially transposed pillar entries
do not necessarily correspond to the same tangent cones.

\begin{ex}
\label{ExLinkDva}
The simplest counterexample to the converse statement that we know is
provided by the following permutations in~$S_6$:
$w=456321$ and $w'=546132$.
Indeed, the corresponding rank matrices are:

\begin{figure}[H]
\begin{center}
\includegraphics[height=3.5cm]{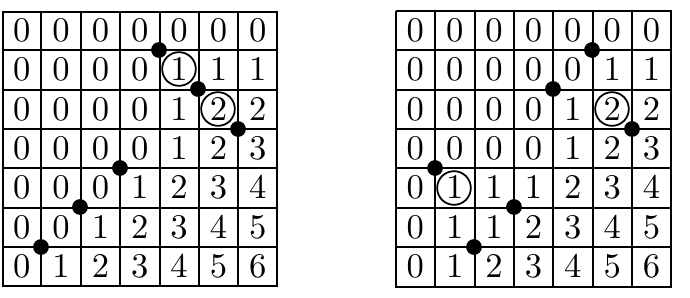}
\end{center}
\caption{Permutations of different length with transposed pillars.}
\label{linkedpilNT}
\end{figure}

\noindent
respectively.
The pillar entries are (partially) transposed, but the permutations
have different length: $\ell(w)=12$ and $\ell(w')=11$,
so that their tangent cones have different dimensions,
and cannot coincide.

Note however the following interesting inclusion: $\T_{w'}\subset\T_w$.
\end{ex}

Another restriction for partial transposition of pillars occurs more often than the above
discussed one.
Given a permutation $w$ and the corresponding rank matrix $r(w)$, 
then a partial transposition of the pillar entries may not correspond to any
rank matrix of any permutation.

\begin{ex}
\label{ExLink}
Consider the permutation $w=34521$ in $S_5$.
The corresponding rank matrix is as follows:
$$
\begin{array}{|c|c|c|c|c|c|}
\hline
0&0&0&0&0&0\\
\hline
0&0&0&\!\!\!\!{}^{{}^\bullet}\raisebox{.1pt}{\textcircled{\raisebox{-.5pt} {$1$}}}\!\!\!&1&1\\
\hline
0&0&0&1&\!\!\!\!{}^{{}^\bullet}\raisebox{.1pt}{\textcircled{\raisebox{-.5pt} {$2$}}}\!\!\!&2\\
\hline
0&0&0&1&2&\!\!\!\!{}^{{}^\bullet}\,3\\
\hline
0&0&\!\!\!\!{}^{{}^\bullet}\,1&2&3&4\\
\hline
0&\!\!\!\!{}^{{}^\bullet}\,1&2&3&4&5\\
\hline
\end{array}
$$
It turns out that there are no rank matrices
with the following pillar entries:
$$
a)\quad
\begin{array}{|c|c|c|c|c|c|}
\hline
0&0&0&0&0&0\\
\hline
0&&&\!\!\raisebox{.1pt}{\textcircled{\raisebox{-.5pt} {$1$}}}\!\!\!&&1\\
\hline
0&&&&&2\\
\hline
0&&&&&3\\
\hline
0&&\!\!\raisebox{.1pt}{\textcircled{\raisebox{-.5pt} {$2$}}}\!\!\!&&&4\\
\hline
0&1&2&3&4&5\\
\hline
\end{array}
\qquad\hbox{b)}\quad
\begin{array}{|c|c|c|c|c|c|}
\hline
0&0&0&0&0&0\\
\hline
0&&&&&1\\
\hline
0&&&&\!\!\raisebox{.1pt}{\textcircled{\raisebox{-.5pt} {$2$}}}\!\!\!&2\\
\hline
0&\!\!\raisebox{.1pt}{\textcircled{\raisebox{-.5pt} {$1$}}}\!\!\!&&&&3\\
\hline
0&&&&&4\\
\hline
0&1&2&3&4&5\\
\hline
\end{array}
$$
Indeed, the above positions of pillar entries are impossible,
since they contradict formula~(\ref{ComPil}), see also Section~\ref{PosPilSec}
for more details.
\end{ex}

\subsection{More on partial transpositions}
Let us briefly discuss the partial transpositions of linked pillar entries.
If one transpose some pillar entries of a rank matrix $r(w)$, but not all of them,  
then the following three possibilities may occur:

1)  there exists a rank matrix of a permutation $w'$
that does have the given set of pillar entries, 
but of different length (cf. Example~\ref{ExLinkDva});

2) there is no rank matrix of a permutation
that has this set of pillar entries
(cf. Example~\ref{ExLink});

3) the ``good case'' where the resulting matrix is a rank matrix of a permutation
that has the given set of pillar entries and the same tangent cone as $w$.

\bigbreak

In view of Conjecture~\ref{TheConjThm} and the above discussion,
the main goal of this paper is to investigate which (partial) transpositions
of pillar entries of a rank matrix $r(w)$ lead to a new permutation $w'$
and do not change the tangent cone.

\section{Admissible partial transpositions: the main theorem}\label{FMTSec}

In this section we describe classes of permutations in $S_n$
with identical tangent cones.
Given a permutation $w$, we define a series of
operations called ``admissible partial transpositions'' and an equivalence class in $S_n$
that consist of permutations related by such transpositions.
We formulate our main result that all permutations from such a class correspond to the same tangent cone.

However, the described classes are not maximal.
Examples in the end of the section show that there are more permutations
with identical tangent cones.

\subsection{Linked and dissociated pillar entries}
We define an equivalence relation on the set of pillar entries of a rank matrix.
Roughly speaking, two pillar entries are in the same class if they are ``close enough'' to each other.

\begin{defi}\label{linked}
(i)
Given a permutation $w\in{}S_n$,
and let $r_{ij}$ and $r_{i'j'}$ be two pillar entries in the rank matrix $r_{ij}(w)$. 
These pillar entries are called \textit{related} if the intervals:
$$
\left[\min{(i,j)},\,\max{(i,j)}\right]
\qquad\hbox{and}\qquad
\left[\min{(i',j')},\,\max{(i',j')}\right]
$$
have a common (real) interior point. 
\begin{figure}[H]
\begin{center}
\includegraphics[height=3.5cm]{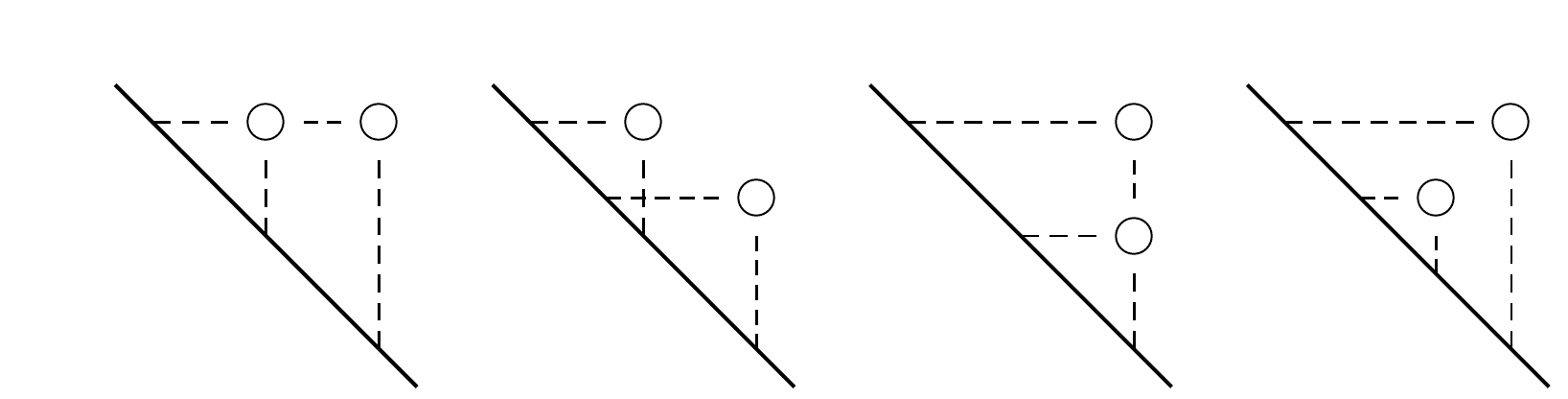}
\end{center}
\caption{Configurations for two related pillars.}
\label{linkedpil}
\end{figure}

(ii)
Pillar entries are called \textit{linked} if they
can be connected by chain of related entries.

(iii)
Otherwise the pillar entries are called \textit{dissociated}.
\begin{figure}[H]
\begin{center}
\includegraphics[height=3.5cm]{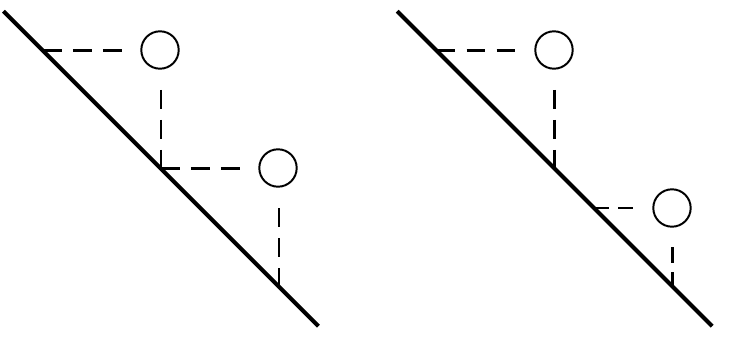}
\end{center}
\caption{Configurations for two dissociated pillars.}
\label{dissocpil}
\end{figure}
\end{defi}

\begin{ex}\label{expillar}
The following rank matrix (in which we omit the extremal rows and columns) corresponding to 
the permutation $w=12,2,9,7,6,4,10,5,3,11,1,8\in S_{12}$:
\begin{figure}[H]
\begin{center}
\includegraphics[height=6cm]{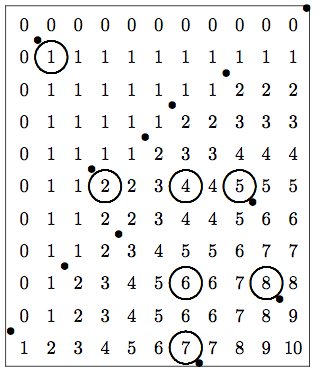}
\end{center}
\caption{Permutation with three classes of linked pillars.}
\label{ExinS12}
\end{figure}

\noindent
 have three classes of linked pillar entries:
$$
\{r_{22}=1\}, \;\{r_{64}=2\},\; \{r_{67}=4, \,r_{69}=5, \,r_{97}=6, \,r_{9,10}=8, \,r_{11,7}=7\}.
$$
\end{ex}

\subsection{The linking graph of pillar entries}

It is convenient to display the linking relations between the pillar entries using a graph. 

\begin{defi}
\label{GraphDef}
The {\it linking graph} is defined as follows.
\begin{enumerate}
\item
The set of vertices of the linking graph is the set of pillar entries of the rank matrix;
\item
two vertices are connected by an edge whenever the corresponding pillar entries are 
related, cf. item (i) of Definition \ref{linked}. 
\end{enumerate}
\end{defi}

For instance, Example~\ref{expillar} corresponds to the following graph

\begin{figure}[H]
\begin{center}
\includegraphics[height=1.5cm]{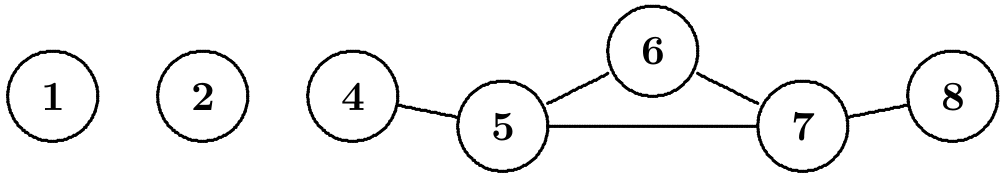}
\end{center}
\caption{Graph of related pillars.}
\label{dissocpil1}
\end{figure}
\noindent
where we have labeled the vertices by the values of the pillar entries (omitting the positions as the values are all different).

The connected components of the linking graph correspond to the classes
of linked pillar entries.

\subsection{Admissible partial transpositions}\label{admpSect}
An {\it admissible partial transposition} is an operation defined
on rank matrices and on the group $S_n$.
Roughly speaking, it consists in transposition of a part of the pillar entries,
such that linked pillar entries transpose (or not) simultaneously.
More precisely, we have the following:

\begin{defi}
\label{ParDef}
Two rank matrices, $r(w)$ and $r(w')$,
are \textit{admissibly partially transpose} to each other if there exists a set $\Lc$
which is a union of classes of linked pillar entries for $r(w)$,
such that the set of pillar entries of $r(w')$ is as follows
\begin{equation}
\label{ParTrEq}
\left\{
\begin{array}{rcllll}
r_{ij}'&=&r_{ij},& & \text{whenever} &r_{ij} \not\in \Lc,\\[4pt]
r_{ji}'&=&r_{ij},& & \text{whenever} &r_{ij}  \in \Lc.\\
\end{array}
\right.
\end{equation}
\end{defi}

\begin{ex}
The permutation $w=11,2,9,8,6,4,5,12,3,7,10,1$ in $S_{12}$ corresponding to the rank matrix 
\begin{figure}[H]
\begin{center}
\includegraphics[height=6cm]{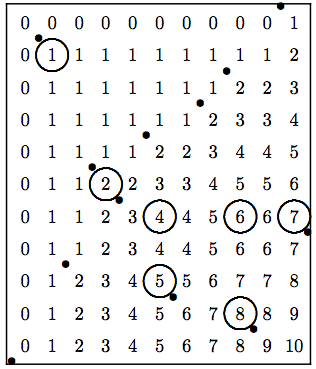}
\end{center}
\caption{An example of admissible partial transposition.}
\label{ExAdmTr}
\end{figure}

\noindent
is admissibly partially transpose to the permutation given in Example \ref{expillar}.
Indeed, the set of pillar entries is the same except for the last 
connected component of the graph, 
for which the positions of the pillar entries are transposed.
\end{ex}

\subsection{Statement of the main theorem}\label{MThSec}

In this section we formulate a sufficient condition 
for the tangent cones of two Schubert varieties to coincide.
Furthermore,
it turns out that every partial transposition of the pillar entries in the associated rank matrices
defines an operation on the group $S_n$.

Our main result is the following

\begin{thm}\label{TheThm}
(i)
Given a permutation $w\in{}S_n$ and the corresponding rank matrix
$r(w)$, for every admissible partial transposition, $r(w)'$, of $r(w)$
there exists a permutation $w'\in{}S_n$ such that $r(w)'=r(w')$.

(ii)
If $w$ and $w'$ are admissibly partially transpose to each other, then
corresponding Shubert varieties have same tangent cones: $\T_{w}=\T_{w'}$.
\end{thm}

We will prove this theorem in Sections~\ref{ElPTrSec} and~\ref{proofdirect}.

\begin{ex}
The Coxeter elements of $S_4$, see Example~\ref{Exn4}, have the same two dissociated pillar
entries, 
$1$ and 
$2$,
and their positions in the rank matrices
differ by transpositions.
Therefore, the Schubert varieties corresponding to these elements
have the same tangent cone. This statement can be generalized, see below
\end{ex}

\begin{cor}
\label{CoCor}
Schubert varieties corresponding to the Coxeter elements
have the same tangent cone.
\end{cor} 

\begin{proof}
The pillar entries of  Coxeter elements are dissociated and differ by partial transpositions;
see Proposition~\ref{ExCoxSec}.
\end{proof}

We will give an explicit description of the corresponding tangent cone in Section~\ref{SectRRM}.
Note also that the Schubert varieties corresponding to the Coxeter elements
are smooth.
Therefore, Corollary \ref{CoCor} can also be deduced from 
the theorem of Lakshmibai, see \cite{Lak1} that describes the Zariski tangent space.

\subsection{Other admissible transpositions}
Theorem~\ref{TheThm} provides large classes of 
Schubert varieties with identical tangent cones.
However, these classes can be yet larger.
In fact, there are other cases of partial transposition of pillar entries
than those considered above.

\begin{ex}
The permutations
$w=6745321$ and $w'=6753421$ in $S_7$
have the following rank matrices:
\begin{figure}[H]
\begin{center}
\includegraphics[height=3.5cm]{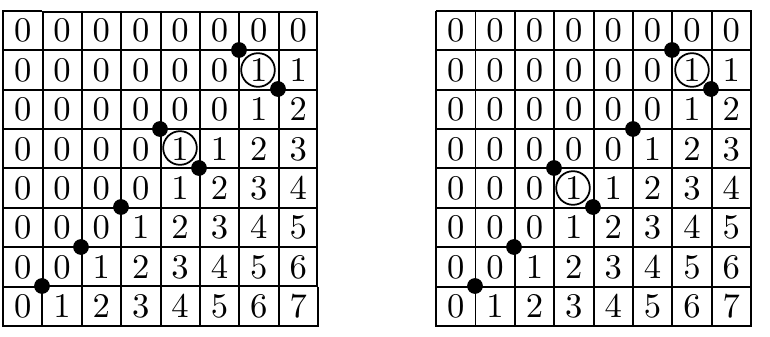}
\end{center}
\caption{An example of admissible transposition of linked pillars.}
\label{linkedpilTOK}
\end{figure}

\noindent
The rank matrix $r(w)$ has two pillar entries:
$r_{16}$ and $r_{34}$, the interval
$[3,4]$ is entirely contained in the interval $[1,6]$.
Therefore,
these pillar entries of $w$ are related
in view of Definition~\ref{linked}.
However, it is easy to check that $\T_w=\T_{w'}$,
in other words, the partial transposition relating $w$ and $w'$
should also be considered as admissible.
\end{ex}

This example is not covered by Theorem~\ref{TheThm} and shows its limits.
For instance, it shows that the converse statement to Part (ii) of the theorem is false.
Existence of such partial transpositions of pillar entries constitutes the main
difficulty in solving the initial classification problem.

\section{Combinatorial aspects of rank matrices and pillar entries}\label{CombSect}

In this section we describe the main properties of pillar entries of rank matrices
and develop the technique necessary from the proof of our main result.

Recall that the set of pillar entries of a rank matrix $r(w)$ determines the permutation~$w$
(see Proposition~\ref{PiP}).
We present two algorithms: that of reconstruction of $w$ from the pillar
entries of $r(w)$, and that of calculating the permutation of $w'$
obtained by some partial transpositions of pillar entries of $r(w)$.
This allows us to prove Part (i) of Theorem~\ref{TheThm}.

We also give an explicit formula for the (co)dimension of the Schubert cell
${\mathcal C}_w$ in terms of the pillar entries of the rank matrix $r(w)$.
This result can be useful for the further study of combinatorics
of rank matrices.

\subsection{Rank matrix and its pillar entries from the
 permutation diagram}\label{PosPilSec}

The rank matrix $r(w)$ is determined by the diagram of the corresponding permutation $w$.

\begin{prop}
\label{Property1}
One has the following formula:
\begin{equation}
\label{dots}
r_{ij}(w)=\#\{\text{\rm dots in the upper left quadrant from the cell } (i,j)\}.
\end{equation}
\end{prop}
\begin{proof}
This readily follows from~(\ref{Ful}).
\end{proof}

The positions of the pillar entries in $r(w)$ can be determined by local structure of the 
diagram of $w$. 
Consider horizontal strips of height~$1$ 
and a vertical strips of width~$1$ in the diagram, such that
the upper left and the lower right corners are marked dots of the permutation:
\begin{center}
\includegraphics[height=3.5cm]{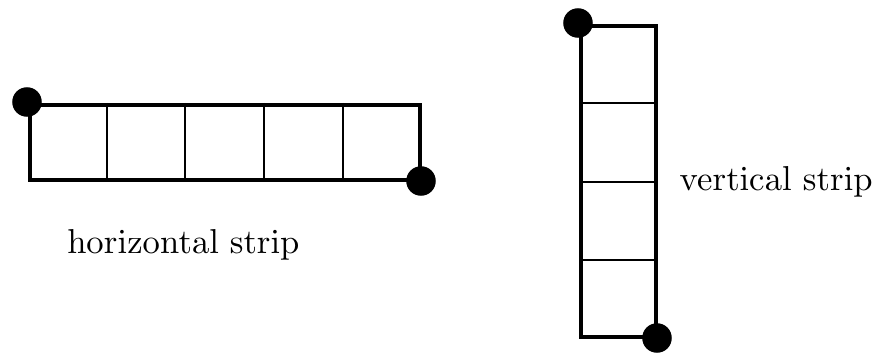}
\end{center}

\begin{prop}
\label{Property2}
Every pillar is located at the intersection of a horizontal strip of height~$1$ 
and a vertical strip of width~$1$.
\end{prop}

\begin{proof}
This is a direct consequence of~(\ref{ComPil}).
\end{proof}


\begin{ex}
Rank matrix and its pillar entries of $w=953471682$ in $S_9$ is as follows.
\begin{figure}[H]
\begin{center}
\includegraphics[width=5cm]{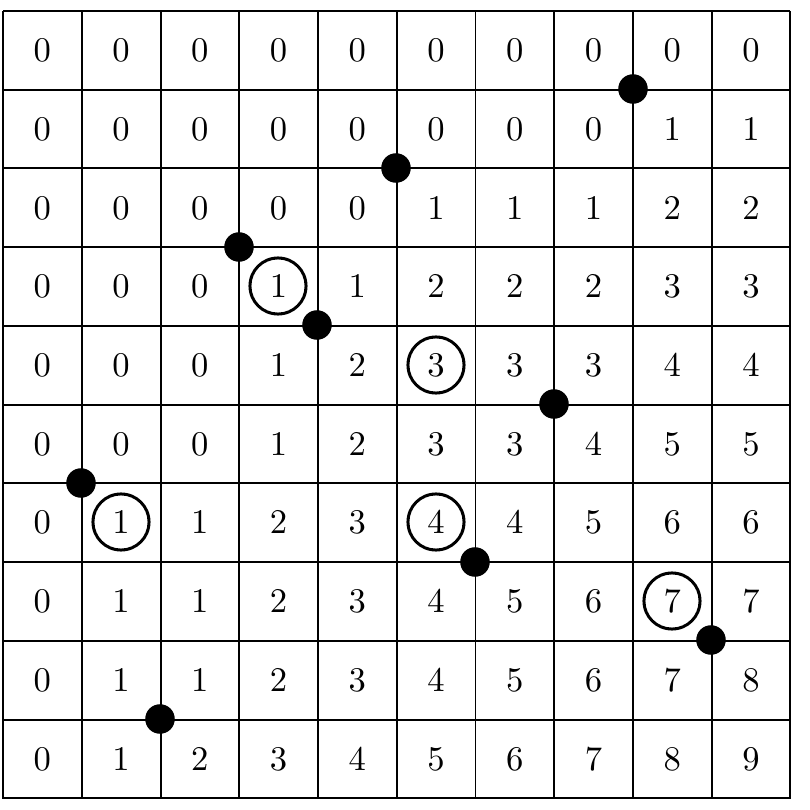}
\includegraphics[width=5cm]{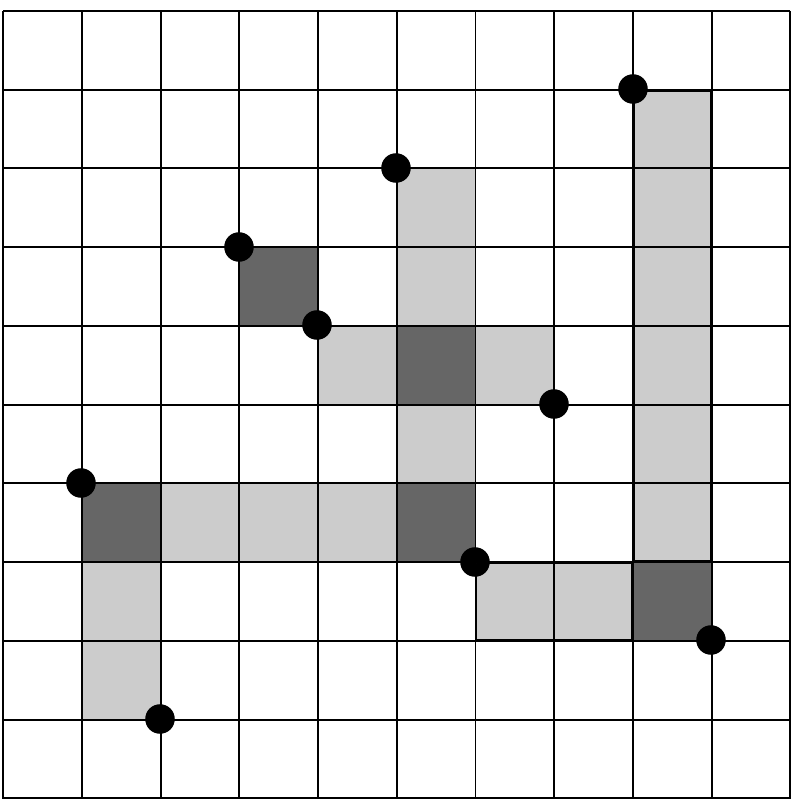}
\end{center}
\caption{Localizing the pillars.}
\label{exmatpil}
\end{figure}
\end{ex}

It will be useful in the sequel to have the following observation.

\begin{prop}
Every horizontal strip of height $1$ 
necessarily intersects with a vertical strip of width $1$, and vice-versa. 
\end{prop}

\begin{proof}
Let $(i,j),\;(i+1,j+k)$, $k>0$, be the marked dots of a horizontal
strip of height one.
If $k=1$, then our horizontal strip is also a vertical strip, and these two
(identical) strips intersect each other.
Let $k>1$, and let $(i_1,j+1),\,\ldots,\,(i_{k-1},j+k-1)$
be the marked dots of the diagram of $w$, lying on the vertical lines crossing our strip;
neither of $i_1,\,\ldots,i_{k-1}$ is $i$ or $i+1$.
If $i_1>i+1$, then the vertical strip with marked dots $(i,j),\,(i_1,j+1)$ intersects
our horizontal strip.
Similarly, if $i_{k-1}<i$, then the vertical strip with marked dots 
$(i_{k-1},j+k-1),\;(i+1,j+k)$ intersects our horizontal strip.
If $i_1<i$ and $i_{k-1}>i+1$, then for some $s$, $i_s<i$ and $i_{s+1}>i+1$.
In this case, the vertical strip with marked dots 
$(i_s,j+s),\,\ldots,\,(i_{s+1},j+s+1)$ intersects our horizontal strip.
\begin{figure}[H]
\begin{center}
\includegraphics[width=4cm]{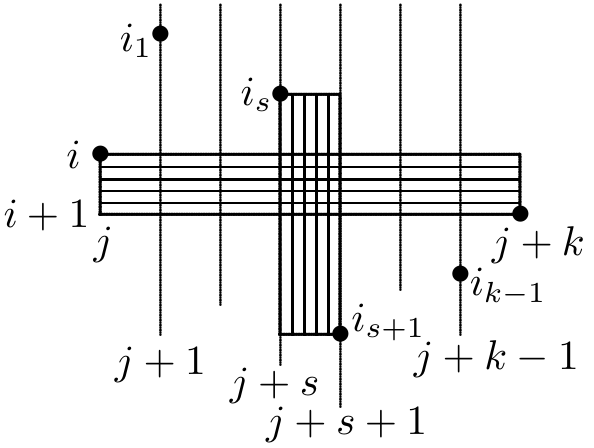}
\end{center}
\label{FigStrip}
\end{figure}
The proof of the vice-versa statement is the same, with the coordinates
of every marked dot switched.
\end{proof}

\subsection{Reconstructing $w$ from the pillar entries of $r(w)$}\label{wfrompil}

In this section we present an algorithm of constructing the diagram of $w$ from the set of pillar entries. 

Let us introduce some useful notation. First we numerate the pillar entries in the 
lexicographical order, that is, from left to right in each row and then counting the rows from top to bottom. Then we set:\smallskip

$(p_i,q_i)=$ the position of the $i$-th pillar;\smallskip

$K_i=r_{p_iq_i}(w)=$ the value of the $i$-th pillar entry;\smallskip

$NW_i=$ The North-west region of the $i$-th pillar entry. \smallskip

We draw the $(n+1)\times(n+1)$ square grid; columns and rows of this grid are separated by~$n$ horizontal lines and~$n$ vertical lines. 
We mark the given pillar entries in $N$ cells of the grid (thus, $N$ is the number of pillar entries). Our permutation $w$ will appear as a set of $n$ dots in the intersections of horizontal and vertical lines, one on each horizontal line and one on each vertical line. 

The diagram of $w$ is constructed in $N+1$ steps. At every step, we place some dots into the interior of the region $NW_i$. If the action requested at any step is impossible by any reason, then our set of ``pillar entries" is not the set of pillar entries of $r(w)$ for any $w$.

For $i=1,\dots,N$, at the $i$-th step, we fist count the number of dots placed in the interior 
of~$NW_i$ at the previous steps. If this number is $L$, we need to add $k_i=K_i-L$ dots into 
$$
NW_i-(NW_1\cup\dots\cup NW_{i-1}).
$$
For this, we numerate the horizontal and vertical lines within $NW_i$ which do not bear any of the $L$ dots placed at the previous steps, respectively from bottom to tor and from right to left. Then, for $j=1,\dots,k_i$ we place a dot at the intersections of the vertical line number $j$ and the horizontal line number $k_i+1-j$. 

\begin{figure}[H]
\begin{center}
\includegraphics[width=2.5cm]{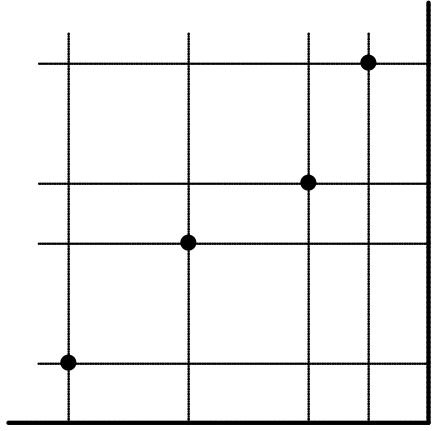}
\end{center}
\caption{Placement of the dots.}
\label{FigAlg}
\end{figure}

\noindent
Our algorithm requests that neither of these dots falls into any of the regions $NW_1,\dots,NW_{i-1}$. 

The final, $(N+1)$-st step works according the same rules with the whole matrix playing the role of $NW_{N+1}$.

The above algorithm is the only way to mark dots without creating an extra pillar  
or changing the values of the pillar entries.
Note that, for the Fulton essential set, a reconstruction algorithm is given in~\cite{EL}.

\begin{ex} 
Figure \ref{algo} below illustrates our algorithm
for $w=853471692\in{}S_9$. 
At each step we color the North-West region at the pillar. 
The dark grey part of the region intersects with North-West regions at previous pillar entries;
the light grey part is the area where the new dots are placed.
\begin{figure}[H]
\begin{center}
\includegraphics[width=4cm]{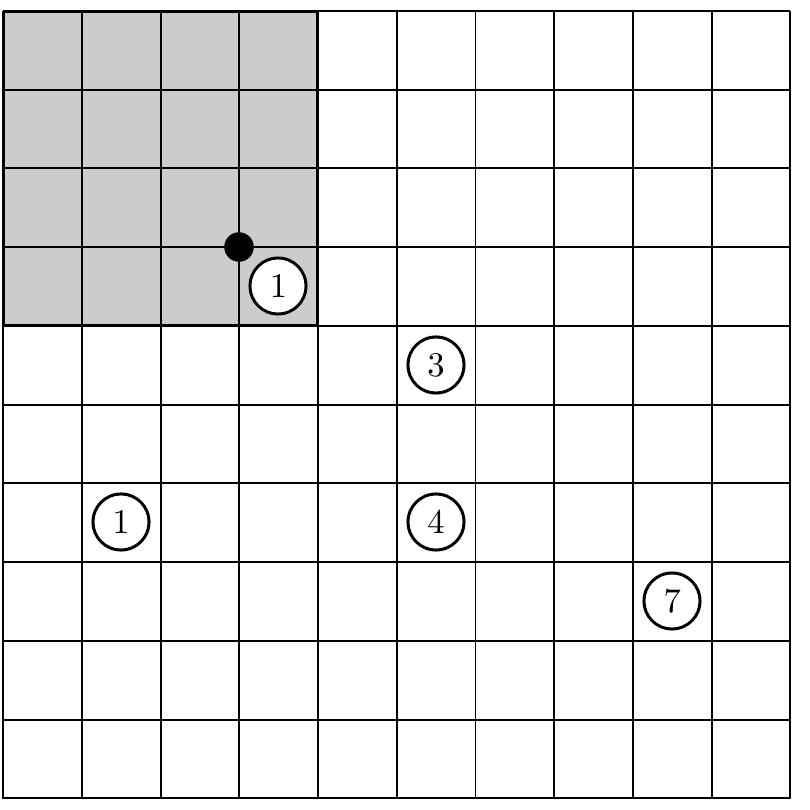}
\includegraphics[width=4cm]{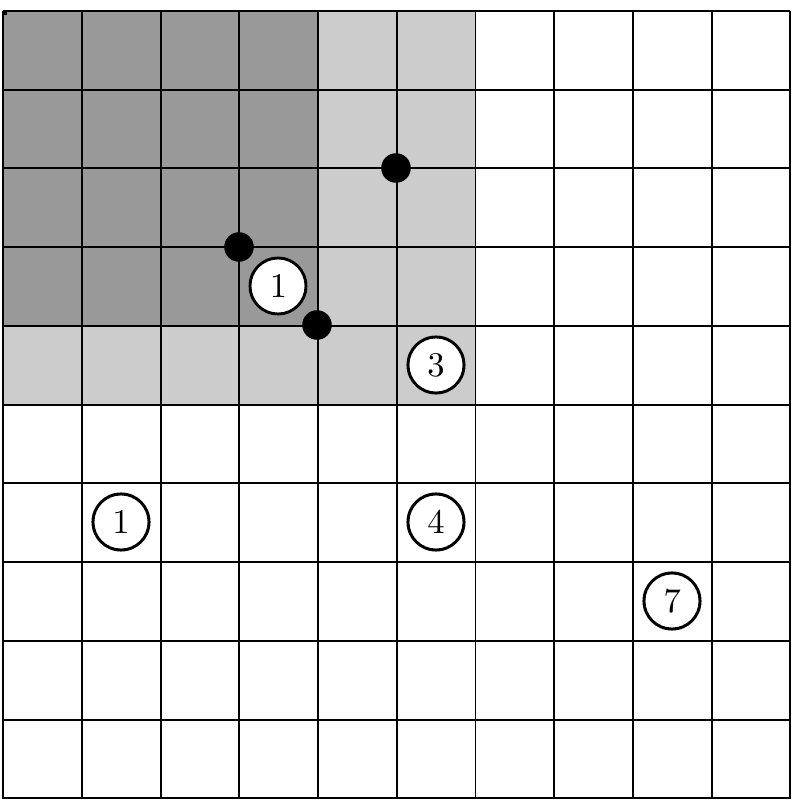}
\includegraphics[width=4cm]{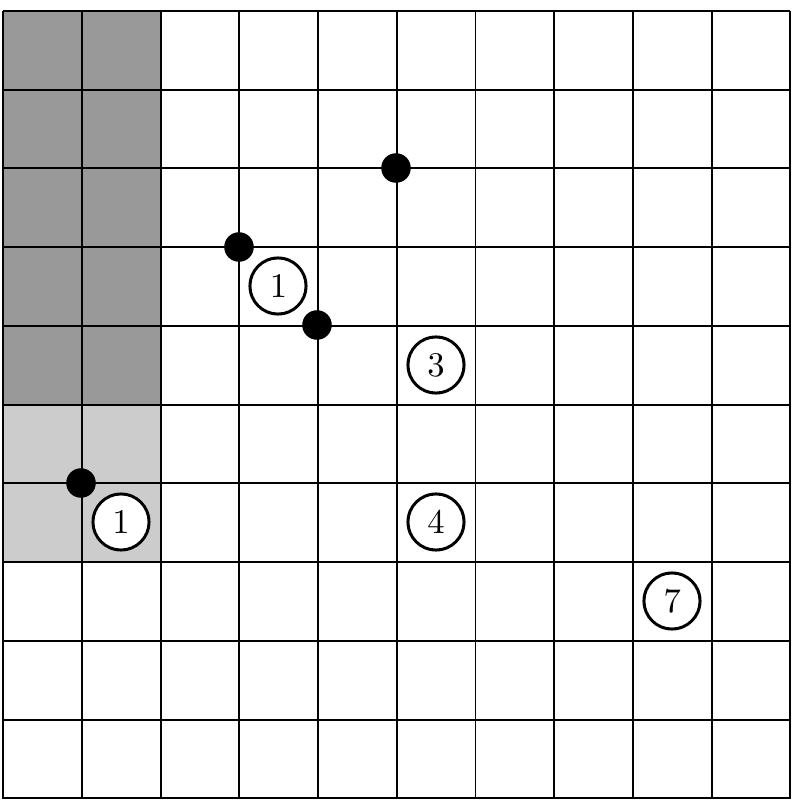}\\
\includegraphics[width=4cm]{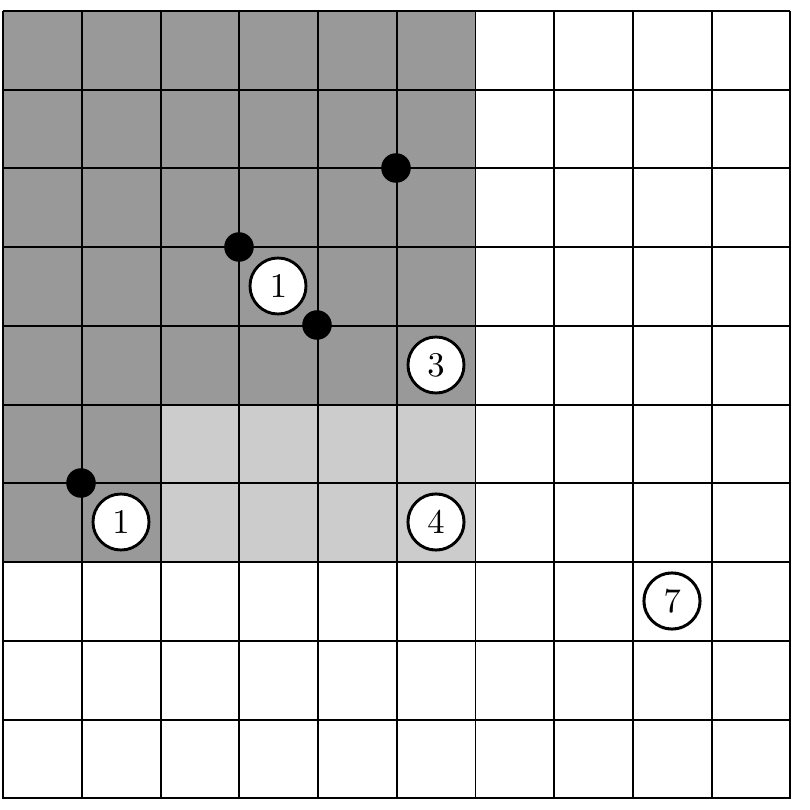}
\includegraphics[width=4cm]{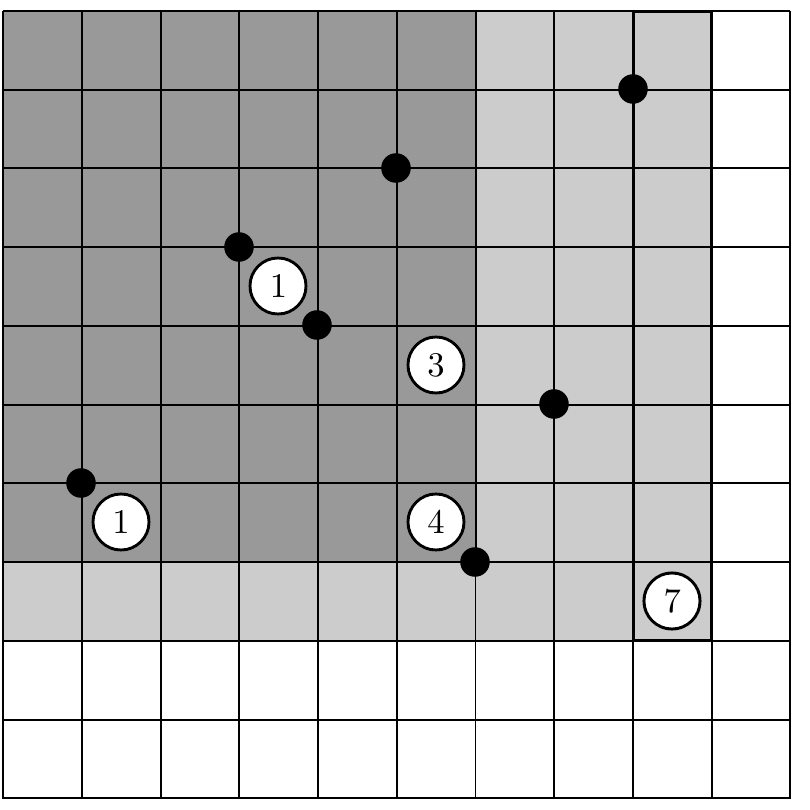}
\includegraphics[width=4cm]{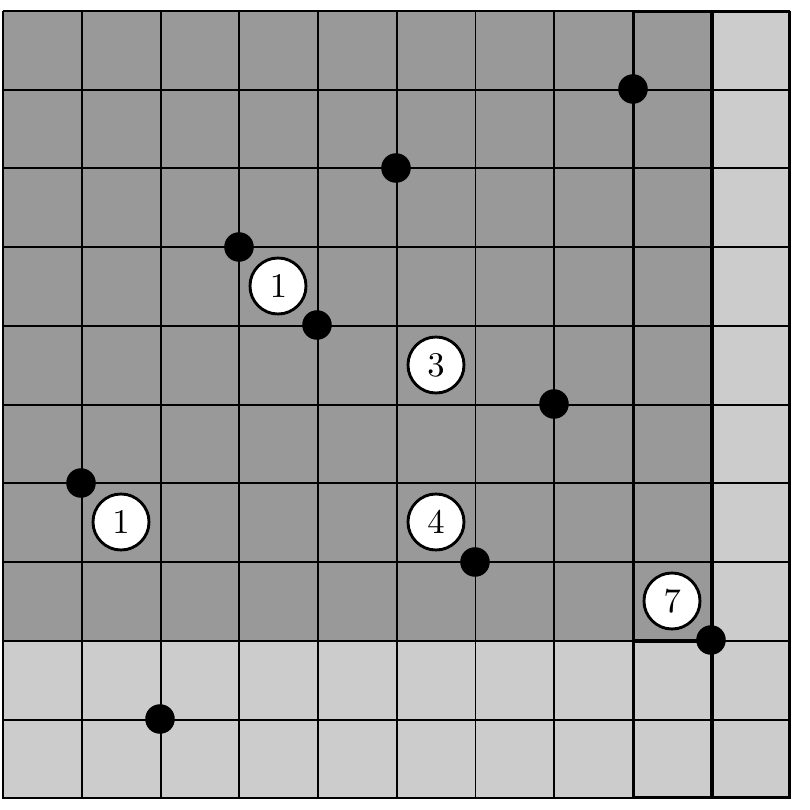}\\
\end{center}
\caption{Recovering the permutation $w$ from the pillar entries.}
\label{algo}
\end{figure}
\end{ex}

\begin{ex} 
Given positive integers, $i,j,a$ satisfying the conditions:
$$
a\leq{}i,j\leq{}n,
\qquad
i+j-a<n,
$$
there exists an element, $w_{i,j,a,n}\in{}S_n$, whose rank matrix
has the unique pillar entry $r_{ij}=a$.
The above algorithm immediately gives:
$$
w_{i,j,a,n}=
(n,n-1,\ldots,n+1-i+a,j,j-1,\ldots,j-a+1,n-i+a,n-i+a-1,\ldots,j+1,j-a,j-a-1,\ldots,1).
$$
This element appeared in~\cite[Corollary~4.5]{RWY}.
\end{ex}

\begin{rem}
If one defines the following partial ordering on the set of ordered pillar entries
\begin{eqnarray*}
j\prec i &\Longleftrightarrow& \text{the $j$-th pillar lies in the region at the North-West of the $i$-th pillar}\\
&&\text{i.e. } j<i  \text{ , } p_j\leq p_i \text{ and } q_j\leq q_i
\end{eqnarray*}
one can write the following relation between the $K_{i}$'s and $k_{i}$'s
$$
K_{i}=k_{i}+\sum_{j\prec i} k_{j}.
$$
\end{rem}

\subsection{(Co)dimension from the set of pillar entries}
The dimension and codimension of a  Schubert cell ${\mathcal C}_w$
(or a Schubert variety $\X_w$)
can be computed directly form the set of pillar entries 
of the corresponding rank matrix $r(w)$. 

The number 
$$
\codim({{\mathcal C}_w})=\ell(ww_0)=\#\{i<j : w(i)<w(j)\}
$$ 
can be obtained in the diagram of $w$ 
counting the intersections of the horizontal segments 
and the vertical segments of the grid 
that are at the right and above each dots, respectively:
\begin{equation}
\label{codimcross}
\codim({{\mathcal C}_w})=\#\{\text{crosses in the diagram of $w$}\},
\end{equation}
 see Figure \ref{crossdim}.

\begin{figure}[H]
\begin{center}
\includegraphics[width=5cm]{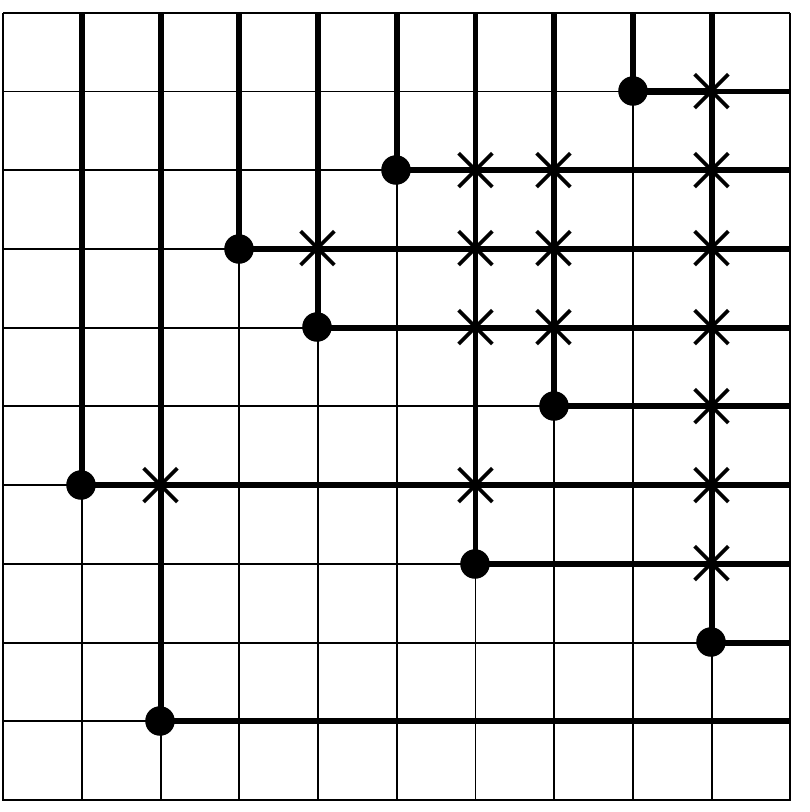}
\end{center}
\caption{$\ell(ww_0)$ from the diagram of $w$.}
\label{crossdim}
\end{figure}

The following formula gives the codimension of a Schubert cell from the data of its pillar entries.

\begin{thm}
\label{codimDBF} 
Using the notation of Section \ref{wfrompil} one computes
$$
\codim({{\mathcal C}_w})=\sum_{i=1}^N\, k_i\; (K_i+n-p_i-q_i).
$$
\end{thm}

\begin{proof}
This formula is obtained using the reconstruction algorithm
of $w$ from the set of pillar entries 
(see Section \ref{wfrompil}) and~\eqref{codimcross}.
 For each dot in the diagram of $w$, we count the crosses on the horizontal segment at its right. 
At step $i$ of the construction, the $k_i$ new dots will contribute with the same number of crosses 
in~\eqref{codimcross}. The reconstruction algorithm of $w$ implies that these crosses can be produced only by the dots that are located at the South-East of the $i$-th pillar (otherwise it would contradict the fact that one uses the closest available vertical lines at the left of the $i$-th pillar). The number of dots in the South-East area is easy to compute from our data:
$$
\begin{array}{rcl}
\#\{\text{dots at SE}\}&=&  \#\{\text{dots}\}-\#\{\text{dots at NW}\}-\#\{\text{dots at NE}\}-\#\{\text{dots at SW}\}\\[4pt]
&=&n-K_i-(p_i-K_{i})-(q_i-K_i)\\[4pt]
&=&K_i+n-p_i-q_i.
\end{array}
$$
Hence the result.
\end{proof}

\subsection{Truncated permutation}

Given a permutation of $w\in{}S_{n}$, 
we will show the existence of permutations whose pillar entries
form the subsets in the set of pillar entries of $r(w)$
obtained by removing of some classes of linked pillar entries.

The pillar entries of $r(w)$ are 
decomposed in the disjoint union of classes of linked pillar entries:
$\{r_{ij}(w)\}=\Lc_{1}\sqcup \Lc_{2}\sqcup \ldots \sqcup \Lc_{s}.$
These classes correspond to subintervals $I_1,I_2,\ldots,I_s$
of the interval $[0,n]$; these subintervals have integer endpoints and pairwise
have no interior points.
The class $\Lc_t$ corresponds to the interval $I_t$, if, for every
$r_{ij}(w)\in\Lc_t$ both $i$ and $j$ belong to~$I_t$.
We order the intervals $I_t$, from the left to the right,
and order the classes $\Lc_t$ accordingly.
Consequently, if $u< v$ and $r_{ij}(w)\in \Lc_{u},\,r_{i'j'}(w)\in \Lc_{v}$,
then $i\leq{}i',\,i\leq{}j',j\leq{}i',\,j\leq{}j'$.
In particular, for $u< v$, all pillar entries from $\Lc_{u}$ lexicographically
precede the pillar entries from $\Lc_{v}$.

The following statement 
is our first application of the reconstruction algorithm 
presented in Section~\ref{wfrompil}.

\begin{prop}
For every $t\in \{1, \ldots, s\}$ there exists a unique permutation, denoted $\tr_{t}(w)$, having 
$\Lc_{1}\sqcup \Lc_{2}\sqcup \ldots \sqcup \Lc_{t}$ as set of pillar entries.
\end{prop}

\begin{proof}
This permutation is obtained by stopping the algorithm of reconstruction of $w$ given in Section \ref{wfrompil} after the step corresponding to the last pillar of the class $\Lc_{t}$ and jumping to the final step.
\end{proof}

\begin{ex}
For $w=(12,2,9,7,6,4,10,5,3,11,1,8)$ as in Example \ref{expillar}, 
the classes are numbered as follows
$$
\Lc_{1}=\{r_{22}=1\}, \;\Lc_{2}=\{r_{64}=2\},\; 
\Lc_{3}=\{r_{67}=4, \,r_{69}=5, \,r_{97}=6, \,r_{9,10}=8, \,r_{11,7}=7\}.
$$
One then obtains the truncated permutations 
$$
\tr_{1}(w)=(12, 2, 3, 11, 10, 9, 8, 7, 6, 5, 4, 1),
\qquad
\tr_{2}(w)=(12, 2, 11, 10, 9, 4, 8, 7, 6, 5, 3, 1).
$$
\end{ex}

\subsection{Elementary partial transpositions}\label{ElPTrSec}
The lexicographic order suggests a natural
series of admissible partial transpositions,
such that all the classes of linked pillar entries $\Lc_i$ transpose
for~$i$ less or equal to some value.
In this section we present en explicit algorithm
of calculating the resulting permutations.
This algorithm is the main ingredient of the proof
of Part (i) of Theorem~\ref{TheThm}.

For $t\in \{1, \ldots, s\}$, 
we define \textit{the elementary partial transposition} $w'=\trp_{t}(w)$, as the  
permutation having the following set of pillar entries:
$$
\left\{
\begin{array}{rclll}
r_{ji}(w')&=&r_{ij}(w), & \text{if} &r_{ij}(w) \in \Lc_{1}\sqcup \ldots \sqcup \Lc_{t};\\[6pt]
r_{ij}(w')&=&r_{ij}(w), & \text{if} &r_{ij}(w)  \in \Lc_{t+1}\sqcup \ldots \sqcup \Lc_{s}.\\
\end{array}
\right.
$$
Note that every partial transposition can be obtained as a sequence of elementary partial transpositions.

Given a permutation $w=w_1w_2\ldots{}w_n\in{}S_n$,
the entries $w_k$ of $w$ are separated into two disjoint groups, $I_1\sqcup{}I_2$:
$$
\left\{
\begin{array}{rclll}
w_k\in{}I_1,& \text{if}&k\leq{}\max(j), & \text{for pillars} 
&r_{ij}(w) \in \Lc_{1}\sqcup \ldots \sqcup \Lc_{t};\\[6pt]
w_k\in{}I_2,& \text{if}&k>{}\min(i), & \text{for pillars} 
&r_{ij}(w)  \in \Lc_{t+1}\sqcup \ldots \sqcup \Lc_{s}.\\
\end{array}
\right.
$$
The algorithm of calculation the permutation~$w'=w'_1w'_2\ldots{}w'_n$,
obtained via the above elementary partial transposition, 
consists in three steps:

\begin{enumerate}
\item
keep $w'_k=w_k\in{}I_1$ if $w_k\leq{}k$, and 
$w'_k=w_k\in{}I_2$ if $w_k\geq{}k$;

\item
inverse the entries $w_k\in{}I_1$, 
i.e., write $k$ at position $w_k$;

\item
fill the remaining positions in $w'$ in the decreasing order.
\end{enumerate}

The proof of the above algorithm is straightforward.

\begin{ex}
For the Coxeter element $w=2341\in{}S_4$, the elementary transposition
$$
\begin{array}{|c|c|c|c|c|}
\hline
0&0&0&0&0\\
\hline
0&0&\!\!\!\!{}^{{}^\bullet}\raisebox{.3pt}{\textcircled{\raisebox{-.9pt} {1}}}\!\!\!&1&1\\
\hline
0&0&1&\!\!\!\!{}^{{}^\bullet}\raisebox{.3pt}{\textcircled{\raisebox{-.9pt} {2}}}\!\!\!&2\\
\hline
0&0&1&2&\!\!\!\!{}^{{}^\bullet}\,3\\
\hline
0&\!\!\!\!{}^{{}^\bullet}\,1&2&3&4\\
\hline
\end{array}
\qquad
\longrightarrow
\qquad
\begin{array}{|c|c|c|c|c|}
\hline
0&0&0&0&0\\
\hline
0&0&0&\!\!\!\!{}^{{}^\bullet}\,1&1\\
\hline
0&\!\!\!\!{}^{{}^\bullet}\raisebox{.3pt}{\textcircled{\raisebox{-.9pt} {1}}}\!\!\!
&1&\!\!\raisebox{.3pt}{\textcircled{\raisebox{-.9pt} {2}}}\!\!\!&2\\
\hline
0&1&1&2&\!\!\!\!{}^{{}^\bullet}\,3\\
\hline
0&1&\!\!\!\!{}^{{}^\bullet}\,2&3&4\\
\hline
\end{array}
$$
is obtained into three steps:
$$
2\,3\,\vert\,4\,1
\quad\to\quad
2\,.\,\vert\,4\,.
\quad\to\quad
.\,1\,\vert\,4\,.
\quad\to\quad
3\,1\,\vert\,4\,2,
$$
so that $w'=3142$ is another Coxeter element, already considered in Example~\ref{Exn4}, b).
\end{ex}

For every $w\in{}S_n$,
the above algorithm implies the existence of
a permutation $w'$ such that the pillar entries of $r(w')$ are obtained by
an admissible partial transposition of pillar entries of $r(w)$.

Part (i) of Theorem~\ref{TheThm} is proved.

\section{Proof of the main theorem}
In this section, we prove Proposition \ref{PiP} and Part (ii) of Theorem \ref{TheThm}.

\subsection{A coordinate system}
In the neighborhood of the standard flag $F_0$,
the flag variety~$\F$ is identified with the subgroup of unitriangular matrices
\begin{equation}
\label{CoOrdEq}
X=
\left(
\begin{array}{cccc}
1&&&\\[4pt]
x_{21}&1&\\[4pt]
\vdots&\ddots&\ddots\\[4pt]
x_{n1}&\cdots&x_{n\,n-1}&\!\!1
\end{array}
\right)
\end{equation}
This defines a local coordinate system $(x_{21},\ldots,x_{n\,n-1})$ on $\F$,
already mentioned in the introduction (see also~\cite{Rya}).
Given a flag $F\in\F$, every space~$V_i$ of $F$ is defined as linear span of
the first $i$ columns of the matrix $X$.

Our next goal is to describe the Schubert cells and Schubert varieties in terms of this
coordinate system.
 
\subsection{Submatrices}\label{SubMS}
Let $M_{ij}$ be the $(n-j)\times{}i$ submatrix
of $X$
consisting of the last $n-j$ rows and the first $i$ columns.

a)
If $i\leq{}j$, then this submatrix is of the form
$$
M_{ij}=
\left(
\begin{array}{ccc}
x_{j+1\,1}&\cdots&x_{j+1\,i}\\
\vdots&&\vdots\\
x_{n1}&\cdots&x_{ni}
\end{array}
\right).
$$

b) If $i>j$, then the submatrix is as follows
$$
M_{ij}=
\left(
\begin{array}{llcc}
x_{j+1\,1}\cdots\; x_{j+1\,j}&1&\\
\vdots&\;\;\;\;\ddots\\
x_{i1}\;\;\;\;\;\cdots&x_{i\,i-1}&1\\
\vdots&&\vdots\\
x_{n1}\;\;\;\;\;\cdots&\cdots&x_{ni}
\end{array}
\right).
$$

\subsection{Relation to the rank matrices}\label{SectRRM}
The following lemma translates the description of Schubert cells
in terms of rank matrices into an algebraic descripiton
in the above coordinate system.

\begin{lem}
\label{DimLem}
The matrix $X$ represents a flag in the Schubert cell ${\mathcal C}_w$ if and only if 
\begin{equation}
\label{DimEq}
\rank(M_{ij})=
i-r_{ij}(w),
\end{equation}
for all $1<i\leq n$, $1\leq j<n$.
Consequently,
the Schubert variety $\X_w$ is described by the conditions: 
$$
\rank(M_{ij})\leq{}i-r_{ij}(w).
$$
\end{lem}

\begin{proof}
The space $\C^j$ consists of vectors with zeros
at positions $\geq{}j+1$.
One then has
$$
j+\rank(M_{ij})=\dim(V_i+\C^j)=i+j-r_{ij}(w).
$$
Hence \eqref{DimEq}.
\end{proof}

\subsection{Systems of equations for $\X_{w}$ and $\T_{w}$}\label{SysEq}
The Schubert variety $\X_w$ is determined,
in a neighborhood of the standard flag $F_0$, by a system of
polynomial equations in the variables~$x_{ij}$.
The equations are obtained as follows. For each couple of indices $i, j$, 
formula \eqref{DimEq} leads to a set of equations
that expresses the annihilation of the minors of the matrix
$M_{ij}$ of size larger than its rank. 
From Proposition~\ref{PiP}, it suffices to consider only the equations for 
the indices $i,j$ corresponding to a pillar entry $r_{ij}(w)$ in the rank matrix of 
 $\X_w$.
 
 The system of equations of the tangent cone $\T_w$ of $\X_{w}$ is
 obtained, roughly speaking, as the homogeneous lower degree parts of the equations of $\X_w$.
 More precisely, the equations of $\X_w$ can be written 
 in such a way that the homogeneous terms of lower degree are linearly independent.
 Then the system of $\T_w$ is obtained by removing all of the monomials of higher degree in the 
 equations of $\X_w$.
 
 \begin{ex}
 As we mentioned in Introduction,
 the first example of a Schubert variety with singularity at the origin correspond
 to the permutation $w=4231\in{}S_{4}$ (see~\cite{Lak,LS}).
 Written in our local coordinates:
 $$
 \left(
 \begin{array}{llll}
 1\\
 x_{21}&1\\
x_{31}&x_{32}&1\\
x_{41}&x_{42}&x_{43}&1
 \end{array}
 \right)
 $$
 the equation of the corresponding tangent cone $\T_{w}$
 (the same as the equation of $\X_{w}$) is: $x_{31}x_{42}-x_{32}x_{41}=0$.
 Indeed, the rank matrix of $w$ is as follows:
 $$
 \begin{array}{|c|c|c|c|c|}
\hline
0&0&0&0&0\\
\hline
0&0&0&0&\!\!\!\!{}^{{}^\bullet}\,1\\
\hline
0&0&\!\!\!\!{}^{{}^\bullet}\raisebox{.3pt}{\textcircled{\raisebox{-.9pt} {1}}}\!\!\!&1&2\\
\hline
0&0&1&\!\!\!\!{}^{{}^\bullet}\,2&3\\
\hline
0&\!\!\!\!{}^{{}^\bullet}\,1&2&3&4\\
\hline
\end{array}
$$
so that the Schubert cell ${\mathcal C}_{w}$ is determined by the condition
$\dim(V_{2}\cap{}\C^{2})=1$, that translates in coordinates
as the condition that a certain linear combination of two first column vectors
belong to the subspace $\C^{2}$, i.e., the matrix $M_{22}$ degenerates.

 The tangent cone~$\T_{w}$ is $5$-dimensional, whereas the
 Zariski tangent space is the whole $6$-dimensional tangent space~$T_{F_{0}}\F$.
  \end{ex}
 
\subsection{The duality}\label{key}
In the case $i\leq{}j$, the minors of $M_{ij}$ are homogeneous polynomial expressions.
The following observation explains the reason for which two pillar entries
transposed to each other, in many situation give the same contribution to the system of equation
of the tangent cones.

If $i>j$, then $M_{ji}$ is the complement of the upper right
square submatrix in $M_{ij}$ (of size~$i-j$) with $1$'s on the diagonal:
$$
M_{ij}=
\left(
\begin{array}{ll|ccc}
&&1&\\
&&\vdots&\ddots\\
&&&\cdots&1\\
\hline
&M_{ji}\\
\end{array}
\right).
$$
The lower degree homogeneous part in the expression of any minors of $M_{ij}$ of size $r\geq i-j$ involving the last $i-j$ columns corresponds
precisely to a minor of $M_{ji}$ of size $r-i+j$, and vice versa.

\subsection{Proof of Proposition \ref{PiP}}\label{ProProS}

Let us show that the pillar entries determine the rank matrix.
We use the fact that the rank matrix $r(w)$ completely determines the Schubert
variety $\X_{w}$. 
 
For a permutation $w$, let $r_{i_1j_1},\ldots,r_{i_Nj_N}$ be the pillar entries of the matrix $r(w)$ and let~$C_{ij}$ be the condition 
$$
\rank(M_{ij})\le i-r_{ij}(w)
$$
from the system of conditions determining the Schubert variety $\X_{w}$ (see Section 5.3). There are obvious implications:\smallskip
{\it
If $r_{i+1,j}=r_{ij}$, then $C_{ij}$ implies $C_{i+1,j}$;

If $r_{i+1,j}=r_{ij}+1$, then $C_{i+1,j}$ implies $C_{ij}$;

If $r_{i,j+1}=r_{ij}$, then $C_{i,j}$ implies $C_{i,j+1}$;

If $r_{i,j+1}=r_{ij}+1$, then $C_{i,j+1}$ implies $C_{ij}$.}\smallskip

Let us visualize the matrix $[C_{ij}]$ as an $(n+1)\times(n+1)$ grid and show the above implications by arrows between the neighboring cells; the resulting diagram for the matrix from 
Example~\ref{ExLink} is shown below.  
\begin{figure}[H]
\begin{center}
\includegraphics[width=4cm]{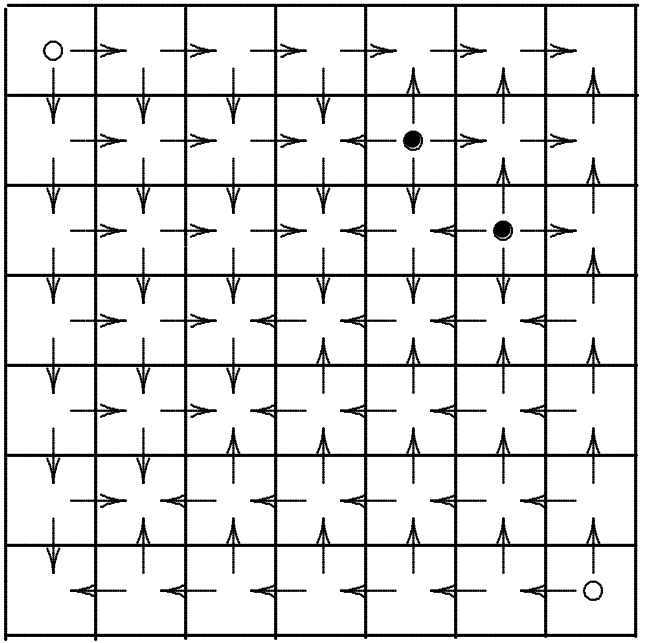}
\end{center}
\caption{The grid of implications.}
\label{TheGrid}
\end{figure}

\noindent
The cells of these grid with only outgoing arrows correspond to the pillars (marked by heavy dots in the picture) and two corner cells: upper left and lower right (marked by light dots). 

Take any cell of the grid and trace a path from it in the following way: we move in the direction opposite to the arrow. 
If there arises a choice of several such direction, choose any of them. If there is no such directions, then stop. 
It is important that our path never passes through a cell more that once: when we move right or upward, the entry stays unchanged, when we move left or downward, the entry grows (by one at a step). 
To move from a cell back to the same cell, we have to make at least one move left of downward, and the entry cannot remain unchanged. 
Thus our path leads from our (arbitrarily chosen) cell to either a pillar or to a corner cell. Moving along this path in the opposite direction, we show that every condition $C_{ij}$ follows either from one of the pillar conditions $C_{i_1j_1},\dots,C_{i_Nj_N}$ or one of the conditions $C_{0 0}, C_{nn}$, which are empty. 
Thus, the Schubert variety $\X_{w}$ is determined by the conditions $C_{i_1j_1},\dots,C_{i_Nj_N}$, that is, determined by the pillars.

To illustrate the above technique, let us calculate the tangent 
cone of the Schubert varieties corresponding to the Coxeter elements explicitly.

\begin{prop}
\label{expCoxProp}
The tangent cone of the Schubert varieties corresponding to the Coxeter elements of $S_n$
is given by the equations
\begin{equation}
\label{CoxTVEq}
x_{ij}=0,
\qquad\hbox{for}\quad
i-j>1.
\end{equation}
\end{prop}

\begin{proof}
It was already proved that 
Schubert varieties corresponding to the Coxeter elements
have the same tangent cone, see Corollary~\ref{CoCor}.
Let us consider the particular Coxeter element 
$$
w=s_1s_2\,\cdots\,{}s_{n-1}=234\ldots{}n1.
$$
Its rank matrix has the following pillar entries (cf. Proposition~\ref{ExCoxSec}):
$$
r_{12}=1,\qquad
r_{23}=2,\quad
\ldots,\quad
r_{n-2,n-1}=n-2.
$$
By Lemma~\ref{DimLem} and Proposition~\ref{PiP},
the Schubert variety $\X_{w}$ is locally determined by the conditions
$\rank(M_{ij})=0$, for~$i>j$.
Therefore, in the local coordinate system $(x_{ij})$, the Schubert variety~$\X_{w}$,
and thus its tangent cone $\T_w$,
is a linear subspace given by the equations~(\ref{CoxTVEq}).
\end{proof}

\subsection{Proof of Theorem~\ref{TheThm}}\label{proofdirect}
We will need the following lemma\footnote{
We are grateful to M. Kashiwara for a simple proof.}.

\begin{lem}
For every $w\in{}S_n$, the Schubert varieties
$\X_{w}$ and $\X_{w^{-1}}$ have same tangent cone.
\end{lem}

\begin{proof}
The homeomorphism $x\mapsto x^{-1}$ from $BwB$ to $Bw^{-1}B$ induces the isomorphism $f\mapsto -f$ from $\T_{w}$ to $\T_{w^{-1}}$.
Since $-f\in\T_w$ for every $f\in\T_w$, this shows that $\T_{w^{-1}}=\T_{w}$.
\end{proof}

We are ready to prove Theorem~\ref{TheThm}, Part (ii).
Assume that two permutations,
$w$ and $w'$, are admissibly partially transpose to each other. 
We want to show that the tangent cones of $\X_{w}$ and $\X_{w'}$ coincide.

We can assume that $w'=\trp_{t}(w)$ is an elementary partial transposition of $w$, 
see Section~\ref{ElPTrSec} for the definition and the notation. 
The systems of equations for $\X_{w}$ and $\X_{w'}$ split in two parts: the equations coming from the pillar entries in the classes
$\Lc_{1}\sqcup \ldots \sqcup \Lc_{t}$ (these equations are \textit{a priori} different for $w$ and $w'$ since the pillar entries are not in the same positions) and the equations coming from the pillar entries in the other classes, namely in
$\Lc_{t+1}\sqcup \ldots \sqcup \Lc_{s}$.
The latter equations are identically the same for $w$ and $w'$. 

Consider finally the two subsystems of equations for $\X_{w}$ and $\X_{w'}$ coming from 
the pillar entries in the set
$\Lc_{1}\sqcup \ldots \sqcup \Lc_{t}$. 
These two subsystems are precisely those 
describing the Schubert varieties associated to 
$\tr_{t}(w)$ and  $\tr_{t}(w')$,
respectively. 
These two varieties have same tangent cones since $\tr_{t}(w)=\tr_{t}(w')^{-1}$. 
After intersecting with the tangent cone of the variety described by the rest of the system, one therefore obtains the same tangent cone for $\X_{w}$ and $\X_{w'}$.

Theorem~\ref{TheThm} is proved.

\section{Enumerative results}

Theorems~\ref{TheThm} gives an efficient method
for calculating the number of different tangent cones
of Schubert varieties.
In this section, we give the result in small dimensions and codimensions.
The general result is still out of reach.

Recall that the total number of Schubert varieties is $n!$,
the total number of their tangent cones is smaller.
It would be interesting to find asymptotic of the number
of tangent cones. 

\subsection{Low-dimensional cases}
In the case $n=4$, the comparative number of Schubert varieties
and their tangent cones, as a function of their dimension, is given by the following table.
$$
\begin{array}{r|c|c|c|c|c|c|c}
\dim&0&1&2&3&4&5&6\\[2pt]
\hline
\hbox{Schub}&1&3&5&6&5&3&1\\[2pt]
\hline
\hbox{TangCones}&1&3&3&3&3&2&1
\end{array}
$$
The total number of tangent cones in this case is $16$.

For $n=5$, the table is:
$$
\begin{array}{r|c|c|c|c|c|c|c|c|c|c|c}
\dim&0&1&2&3&4&5&6&7&8&9&10\\[2pt]
\hline
\hbox{Schub}&1&4&9&15&20&22&20&15&9&4&1\\[2pt]
\hline
\hbox{TangCones}&1&4&6&7&9&9&10&8&6&2&1
\end{array}
$$
The total number of tangent cones is $63$.

For $n=6$, the distribution of the tangent cones is as follows:
$$
\begin{array}{r|c|c|c|c|c|c|c|c|c|c|c|c|c|c|c|c|c}
\dim&0&1&2&3&4&5&6&7&8&9&10&11&12&13&14&15\\[2pt]
\hline
\hbox{TangCones}&1&5&10&14&20&25&31&36&40&40&34&24&15&8&3&1
\end{array}
$$
The total number of tangent cones for $n=6$ is $343$.

For $n=7$ and $8$, the total numbers of tangent cones are:
$1821$ and $13041$, respectively\footnote{These numbers are obtained using computer programs.}.
Note that the sequence $16,63,343,1821,13041,\ldots$
does not appear in Sloane's online Encyclopedia of Integer Sequences.

\subsection{Tangent cones of codimension $2$}

Let us also consider the case of small codimension.

The tangent cone of the Schubert variety $\X_{w_0}$ corresponding to longest element $w_0\in{}S_n$,
is the only one tangent cone of dimension $\frac{n(n-1)}{2}$.

Next, in the case of dimension $\frac{n(n-1)}{2}-1$ ({\it i.e.}, of codimension $1$),
there are $n-1$ Schubert varieties that have $\left[\frac{n}{2}\right]$ tangent cones.
Indeed, the elements $\X_w$ and $\X_{w^{-1}}$ have the same tangent cone.

There are $\frac{(n+2)(n-1)}{2}$ Schubert varieties of codimension $2$.
The number of their tangent cones depend on the parity of $n$, as given by the following
statement.

\begin{prop}
\label{EnumProp}
The number of tangent cones of codimension $2$ is:
$$
2+\frac{(n-3)(n+11)}{8},
\qquad\hbox{and}\qquad
3+\frac{(n-4)(n+14)}{8},
$$
for odd $n$, and for even $n$, respectively.
\end{prop}

\begin{proof}
A straightforward calculation.
\end{proof}

\section*{Appendix}

\subsection*{A1 Comparison of pillar entries to essential entries}
Below are a series of examples 
and comments about the relationship between pillar entries and Fulton's essential entries,
see also~\cite{Woo}.
Recall that essential entries are boxed (while pillar entries are encircled as above).

Let us consider examples that emphasize the difference between
the notions of essential and pillar entries.
The most interesting case is that of the Coxeter elements.

\begin{ex}
\label{BothEx}
a)
The rank matrix of the element $w_0=4\,3\,2\,1$
in $S_4$ has three essential entries
$$
\begin{array}{|c|c|c|c|c|}
\hline
0&0&0&0&0\\
\hline
0&0&0&\!\mybox{0}\!\!&\!\!\!\!{}^{{}^\bullet}\,1\\
\hline
0&0&\!\mybox{0}\!\!&\!\!\!\!{}^{{}^\bullet}\,1&2\\
\hline
0&\!\mybox{0}\!\!&\!\!\!\!{}^{{}^\bullet}\,1&2&3\\
\hline
0&\!\!\!\!{}^{{}^\bullet}\,1&2&3&4\\
\hline
\end{array}
$$
and no pillar entries.
It can be deduced from formula~(\ref{Ful}), that,
for an arbitrary $n$,
the only rank matrix without pillar entries is
the matrix $r(w_0)$ of the longest element $w_0\in{}S_n$.
This matrix has $n-2$ essential entries along the antidiagonal.

b)
For each of the elements $w_1=2\,1\,4\,3$ and $w_2=4\,2\,3\,1$ of $S_4$,
we have two essential entries and one pillar:
$$
\begin{array}{|c|c|c|c|c|}
\hline
0&0&0&0&0\\
\hline
0&\!\mybox{0}\!\!&\!\!\!\!{}^{{}^\bullet}\,1&1&1\\
\hline
0&\!\!\!\!{}^{{}^\bullet}\,1&\!\!\raisebox{.3pt}{\textcircled{\raisebox{-.9pt} {2}}}\!\!\!&2&2\\
\hline
0&1&2&\!\mybox{2}\!\!&\!\!\!\!{}^{{}^\bullet}\,3\\
\hline
0&1&2&\!\!\!\!{}^{{}^\bullet}\,3&4\\
\hline
\end{array}
\qquad\hbox{and}\qquad
\begin{array}{|c|c|c|c|c|}
\hline
0&0&0&0&0\\
\hline
0&0&0&\!\mybox{0}\!\!&\!\!\!\!{}^{{}^\bullet}\,1\\
\hline
0&0&\!\!\!\!{}^{{}^\bullet}\raisebox{.3pt}{\textcircled{\raisebox{-.9pt} {1}}}\!\!\!&1&2\\
\hline
0&\!\mybox{0}\!\!&1&\!\!\!\!{}^{{}^\bullet}\,2&3\\
\hline
0&\!\!\!\!{}^{{}^\bullet}\,1&2&3&4\\
\hline
\end{array}
$$
Note that the position of the pillar entry in the above matrices is the same,
while those of the essential entries are different.

c)
For the Coxeter elements of $S_{4}$, we have:
$$
\begin{array}{|c|c|c|c|c|}
\hline
0&0&0&0&0\\
\hline
0&0&\!\!\!\!{}^{{}^\bullet}\raisebox{.3pt}{\textcircled{\raisebox{-.9pt} {1}}}\!\!\!&1&1\\
\hline
0&0&1&\!\!\!\!{}^{{}^\bullet}\raisebox{.3pt}{\textcircled{\raisebox{-.9pt} {2}}}\!\!\!&2\\
\hline
0&\!\mybox{0}\!\!&1&2&\!\!\!\!{}^{{}^\bullet}\,3\\
\hline
0&\!\!\!\!{}^{{}^\bullet}\,1&2&3&4\\
\hline
\end{array}
\qquad
\begin{array}{|c|c|c|c|c|}
\hline
0&0&0&0&0\\
\hline
0&0&\!\!\!\!{}^{{}^\bullet}\raisebox{.3pt}{\textcircled{\raisebox{-.9pt} {1}}}\!\!\!&1&1\\
\hline
0&\!\mybox{0}\!\!&1&\!\mybox{1}\!\!&\!\!\!\!{}^{{}^\bullet}\,2\\
\hline
0&\!\!\!\!{}^{{}^\bullet}\,1&\!\!\raisebox{.3pt}{\textcircled{\raisebox{-.9pt} {2}}}\!\!\!&2&3\\
\hline
0&1&2&\!\!\!\!{}^{{}^\bullet}\,3&4\\
\hline
\end{array}
\qquad
\begin{array}{|c|c|c|c|c|}
\hline
0&0&0&0&0\\
\hline
0&0&\!\mybox{0}\!\!&\!\!\!\!{}^{{}^\bullet}\,1&1\\
\hline
0&\!\!\!\!{}^{{}^\bullet}\raisebox{.3pt}{\textcircled{\raisebox{-.9pt} {1}}}\!\!\!
&1&\!\!\raisebox{.3pt}{\textcircled{\raisebox{-.9pt} {2}}}\!\!\!&2\\
\hline
0&1&\!\mybox{1}\!\!&2&\!\!\!\!{}^{{}^\bullet}\,3\\
\hline
0&1&\!\!\!\!{}^{{}^\bullet}\,2&3&4\\
\hline
\end{array}
\qquad
\begin{array}{|c|c|c|c|c|}
\hline
0&0&0&0&0\\
\hline
0&0&0&\!\mybox{0}\!\!&\!\!\!\!{}^{{}^\bullet}\,1\\
\hline
0&\!\!\!\!{}^{{}^\bullet}\raisebox{.3pt}{\textcircled{\raisebox{-.9pt} {1}}}\!\!\!
&1&1&2\\
\hline
0&1&\!\!\!\!{}^{{}^\bullet}\raisebox{.3pt}{\textcircled{\raisebox{-.9pt} {2}}}\!\!\!&2&3\\
\hline
0&1&2&\!\!\!\!{}^{{}^\bullet}\,3&4\\
\hline
\end{array}
$$
\end{ex}

\medskip

\subsection*{A2 Rothe diagrams and opposite Rothe diagrams}\label{RiTdSec}

The Rothe diagram~\cite{Rot} of a permutation $w\in S_{n}$ 
is an $n\times{}n$ square table obtained according to the following rule.
Dot the cell $(i,j)$ whenever $w(i)=j$, 
shade all the cells of the row at the right of the dotted cell 
and all the cells of the column below the dotted cell 
(including the dotted cell).
Note that the length $\ell(w)$ is equal to the number of white cells in the Rothe diagram.

It was noticed in \cite{Ful1}, that the white cells having a South and East 
frontier with the shaded region give the positions of the essential entries 
in the corresponding rank matrix.
The value of an essential entry is equal to
the number of dots in the upper left quadrant
of the Rothe diagram with the origin at the corresponding cell.
Let us explain a similar rule to obtain positions of pillar entries.

\begin{table}[H]
\begin{tabular}{|c|c|c|c|}
\hline 
&\cellcolor{lightgray}\textbullet&\cellcolor{lightgray}&\cellcolor{lightgray}\\
\hline 
&\cellcolor{lightgray}&\cellcolor{lightgray}\textbullet&\cellcolor{lightgray}\\
\hline 
\!\raisebox{.9pt}{\boxed{\raisebox{.9pt} {\,}}}\!
&\cellcolor{lightgray}&\cellcolor{lightgray}&\cellcolor{lightgray}\textbullet\\
\hline 
\cellcolor{lightgray}\textbullet&\cellcolor{lightgray}&\cellcolor{lightgray}&\cellcolor{lightgray}\\
\hline
\end{tabular}
\qquad
\begin{tabular}{|c|c|c|c|}
\hline 
\cellcolor{lightgray}&\!\!\raisebox{.3pt}{\textcircled{\raisebox{-.9pt} {\textbullet}}}\!\!
&\cellcolor{lightgray}&\cellcolor{lightgray}\\
\hline 
\cellcolor{lightgray}&\cellcolor{lightgray}&\!\!\raisebox{.3pt}{\textcircled{\raisebox{-.9pt} {\textbullet}}}\!\!
&\cellcolor{lightgray}\\
\hline 
\cellcolor{lightgray}&\cellcolor{lightgray}&\cellcolor{lightgray}&\textbullet\\
\hline 
\textbullet&&&\\
\hline
\end{tabular}
\medskip
\caption{The Rothe diagram (left) and the opposite Rothe diagram (right) 
of the Coxeter permutation $2\,3\,4\,1$.
The Rothe diagram gives the unique essential entry in the rank matrix: $r_{31}=0$,
whereas the opposite diagram gives two pillar entries:
$r_{12}=1$ and $r_{23}=2$.}
\end{table}

Consider the {\it opposite Rothe diagram} obtained with the following rule.
Shade all the cells of the row strictlty at the left of the dotted cell 
and all the cells of the column strictly above the dotted cell
(the dotted cell is not shaded).
Note that the number of white undotted cells in the opposite Rothe diagram is equal to $\ell(w)$.

It follows directly from Definition~\ref{PilDef}, that
the white cells having a South and East frontier with the shaded region in the opposite Rothe diagram
give the positions of the pillar entries in the corresponding rank matrix.
The value of a pillar entry is equal to 
the number of dots in the upper left quadrant of the diagram.

\medskip

{\bf Acknowledgements.}
 This work was started during our stay at the American Institute of Mathematics
within the research program SQuaRE.
The final version was completed during a RiP stay at 
the Mathematisches Forschungsinstitut Oberwolfach.
We are grateful to AIM and MFO for their hospitality.
We are also grateful to M. Kashiwara and A. Panov for helpful discussions,
and to the referee for his/her helpful comments.
S. M-G. was partially supported by the ANR project SC$^{3}$A, ANR-15-CE40-0004-01.

\goodbreak


\end{document}